\documentclass[numbers,webpdf,imamci]{ima-authoring-template}%

\usepackage{comment}
\usepackage{amssymb}

\graphicspath{{Fig/}}

\theoremstyle{thmstyletwo}%
\newtheorem{theorem}{Theorem}
\newtheorem{proposition}[theorem]{Proposition}%
\newtheorem{remark}{Remark}%
\newtheorem{lemma}{Lemma}
\newtheorem{assumption}{Assumption}

\newtheorem{definition}{Definition}

\numberwithin{equation}{section}

\renewcommand{\L}{\mathrm{L}}
\renewcommand{\d}{\mathrm{d}}
\renewcommand{\H}{\mathrm{H}}
\newcommand{\rank}{\mathrm{rank}}

\newcommand{\sign}{\mathrm{sign}}

\begin{document}

\DOI{}
\copyrightyear{}
\vol{00}
\pubyear{}
\access{Advance Access Publication Date: Day Month Year}
\appnotes{Paper}
\firstpage{1}

\title[\texorpdfstring{$H^\infty$}{}-control for a class of BC hyperbolic PDEs]{$H^\infty$-control for a class of boundary controlled hyperbolic PDEs}

\author{Anthony Hastir*\ORCID{0000-0002-3776-6339} and Birgit Jacob\ORCID{0000-0002-8775-4424}
\address{\orgdiv{School of Mathematics and Natural Sciences}, \orgname{University of Wuppertal}, \orgaddress{\street{Gau{\ss}stra{\ss}e, 20}, \postcode{42119}, \state{Wuppertal}, \country{Germany}}}}
\author{Hans Zwart\ORCID{0000-0003-3451-7967}
\address{\orgdiv{Department of Applied Mathematics}, \orgname{University of Twente}, \orgaddress{\street{P.O. Box 217}, \postcode{7500 AE}, \state{Enschede}, \country{The Netherlands}}} and \address{\orgdiv{Department of Mechanical Engineering}, \orgname{Eindhoven University of Technology}, \orgaddress{\postcode{5600 MB}, \state{Eindhoven}, \country{The Netherlands}}}} 

\authormark{Anthony Hastir et al.}

\corresp[*]{Corresponding author: \href{email:hastir@uni-wuppertal.de}{hastir@uni-wuppertal.de}}

\received{Date}{0}{Year}
\revised{Date}{0}{Year}
\accepted{Date}{0}{Year}


\abstract{A solution to the suboptimal $H^\infty$-control problem is given for a class of hyperbolic partial differential equations (PDEs). The first result of this manuscript shows that the considered class of PDEs admits an equivalent representation as an infinite-dimensional discrete-time system. Taking advantage of this, this manuscript shows that it is equivalent to solve the suboptimal $H^\infty$-control problem for a finite-dimensional discrete-time system whose matrices are derived from the PDEs. After computing the solution to this much simpler problem, the solution to the original problem can be deduced easily. In particular, the optimal compensator solution to the suboptimal $H^\infty$-control problem is governed by a set of hyperbolic PDEs, actuated and observed at the boundary. We illustrate our results with a boundary controlled and boundary observed vibrating string.}

\maketitle
\section{Introduction}

The $H^\infty$-control has been mainly developed in the '80s and is a paramount optimization problem in control theory. The main idea behind $H^\infty$-control is to design a dynamic controller whose objective is twofold: 1. internally stabilizing the closed-loop system, 2. minimizing the effect of the disturbance on the regulated output. Mathematically speaking, consider the dynamical system $\Sigma$ with input $u$, disturbance $d$ and outputs $z$ and $y$ as depicted in Figure \ref{fig:Diagram_Open_Loop}. In what follows, we will refer to $z$ as the \textit{to-be regulated} output.

\begin{figure}
  \centering
  \includegraphics[scale=1.5]{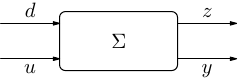}
  \caption{Open-loop system $\Sigma$ with inputs $d$ and $u$ and outputs $z$ and $y$.\label{fig:Diagram_Open_Loop}}
\end{figure}

The objective is to design a dynamic controller, $\Sigma_c$, that takes as input the output $y$ and that returns $u$ as output, see Figure \ref{fig:Diagram_Closed_Loop}. 
\begin{figure}
  \centering
  \includegraphics[scale=1.5]{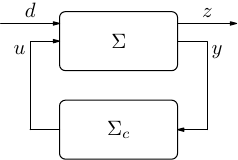}
  \caption{Closed-loop system resulting from the interconnection of $\Sigma$ and $\Sigma_c$.\label{fig:Diagram_Closed_Loop}}
\end{figure}
In $H^\infty$-control, the design of such a controller is done such that the dynamical system resulting from the interconnection is internally stable and such that the closed-loop transfer function from the disturbance $d$ to the \textit{to-be regulated} output $z$ has minimum $H^\infty$-norm. Another version of that control problem is known as suboptimal $H^\infty$-control, which means that the $H^\infty$-norm of the aforementioned transfer function has to be less than some prescribed constant.  

The problem of reduction by feedback of the sensitivity of a system subject to external disturbances has been studied in details in \cite{Zames_81} in which the author uses a weighted semi-norm to measure the sensitivity. In particular, it is shown that uncertainty reduces the ability of feedback to diminish the sensitivity. Few years later, the $H^\infty$-control problem as described above has been characterized in \cite{FrancisDoyle_SIAM_87} and \cite{DoyleGlover_TAC_89} for plants governed by finite-dimensional continuous-time systems. In these references, a parameterization of all the controllers satisfying an $H^\infty$ performance criterion is given. In particular, the solution to that optimization problem is shown to rely on the solution of two Riccati equations. The $H^\infty$-control problem has been studied for finite-dimensional discrete-time systems in \cite{IonescuWeiss_93} in which the authors use the concept of $J$-factorization to characterize the solution. Different variations of the classical $H^\infty$-control problem can be found in the book \cite{Stoorvogel_book_92} in which systems driven by either continuous-time or discrete-time finite-dimensional dynamics are considered. For more results on Riccati equations, $H^\infty$-control and $J$-factorization, we refer the reader to the book \cite{IonescuWeiss_99}. Therein, finite-dimensional systems are considered, both in continuous-time and discrete-time.

$H^\infty$-control has also been studied for infinite-dimensional systems. In particular, the references \cite{Curtain_Hinf_90}, \cite{Curtain_Hinf_91} and \cite{BergelingMorrisRantzer_Aut_20} treat that problem in the case where the input and the output operators are bounded linear operators. In particular, the solution is still shown to be based on the solution of Riccati equations. In addition, an explicit formula for the optimal controller solution to the $H^\infty$-control problem is presented in \cite{BergelingMorrisRantzer_Aut_20}. When the input and output operators are allowed to be unbounded, Staffans developed a theory for solving the suboptimal $H^\infty$-control problem in which spectral factorization is used, see \cite{Staffans_Hinf_98}. Another reference treating this problem in infinite dimensions is \cite{Vankeulen_93}. Therein, both bounded and unbounded control and observation operators are considered. In particular, a solution to the $H^\infty$-control problem is given for systems that fall into the \textit{Pritchard-Salamon} class of distributed parameter systems, see \cite{Salamon_87}.
 
Hyperbolic partial differential equations are an important class of dynamical systems. Their study as infinite-dimensional systems on Banach spaces, and in particular in the context of semigroup theory may be found in the early references \cite{Hyperbolic_Phillips_1957} and \cite{Hyperbolic_Phillips_1959}. Many physical applications in engineering may be modeled by this class of PDEs. As an example, the propagation of current and voltage in transmission lines is modeled by hyperbolic PDEs, known as the telegrapher equations. The change of level of water and sediments in open channels is described by hyperbolic PDEs, called the Saint-Venant-Exner equations. In addition, when looking at fluid dynamics, hyperbolic PDEs are able to describe the motion of an inviscid ideal gas in a rigid cylindrical pipe, via the Euler equations. In a more general sense, conservation laws may be described by hyperbolic PDEs. We refer the reader to \cite{BastinCoron_Book} for a comprehensive overview of applications governed by hyperbolic PDEs. Therein, these equations are presented both in their nonlinear and linear forms, where aspects like stability and control are described as well. 

In this manuscript, we consider the following class of hyperbolic partial differential equations
\begin{align}
  \frac{\partial x}{\partial t}(\zeta,t) &= -\frac{\partial}{\partial \zeta}(\lambda_0(\zeta)x(\zeta,t)),\nonumber\\
  x(\zeta,0) &= x_0(\zeta).\label{StateEquation_Diag_Uniform}
  \end{align}
  The variables $t\geq 0$ and $\zeta\in [0,1]$ are the time and space variables, respectively, whereas $x(\cdot,t)\in\L^2(0,1;\mathbb{R}^n) =: X$ is the state with initial condition $x_0\in X$. The PDEs are complemented with the following inputs and outputs
  \begin{align}
    Ed(t) + \left[\begin{matrix}0\\ I\end{matrix}\right]u(t) &= - K(\lambda_0(0)x(0,t)) - L (\lambda_0(1)x(1,t)),\label{Input_Diag_Uniform}\\
  z(t) &= -K_z(\lambda_0(0)x(0,t)) - L_z(\lambda_0(1)x(1,t)),\label{Output_Reg_Diag_Uniform}\\
  y(t) &= - K_y (\lambda_0(0)x(0,t)) - L_y (\lambda_0(1)x(1,t)),\label{Output_Diag_Uniform}
  \end{align}
  $t\geq 0$, where $u(t)\in\mathbb{R}^p, p\in\mathbb{N}$ is the control input, $d(t)\in\mathbb{R}^{k}, k\in\mathbb{N}$ is a disturbance, $y(t)\in\mathbb{R}^m, m\in\mathbb{N}$ is the measured output and $z(t)\in\mathbb{R}^l, l\in\mathbb{N}$ is the \textit{to-be regulated} output. The scalar-valued function $\lambda_0\in \L^\infty(0,1;\mathbb{R}^+)$ satisfies $\lambda_0(\zeta)\geq \epsilon > 0$ a.e.~on $[0,1]$. The matrices $E, K, L, K_y, L_y, K_z$ and $L_z$ are such that $E\in\mathbb{R}^{n\times k}, K\in\mathbb{R}^{n\times n}, L\in\mathbb{R}^{n\times n}, K_y\in\mathbb{R}^{m\times n}, L_y\in\mathbb{R}^{m\times n}, K_z\in\mathbb{R}^{l\times n}$ and $L_z\in\mathbb{R}^{l\times n}$. 

It is possible to consider the additive term $M(\zeta)x(\zeta,t)$ in the PDEs \eqref{StateEquation_Diag_Uniform}, with $M(\cdot)\in\L^\infty(0,1;\mathbb{R}^{n\times n})$. In that case, it is shown in \cite{HastirJacobZwart_Riesz} and \cite{HastirJacobZwart_LQ} that the change of variables $\tilde{x}(\zeta,t) = Q(\zeta)x(\zeta,t)$ makes the term $M(\zeta)x(\zeta,t)$ disappear\footnote{The matrix $Q$ is the solution to the matrix differential equation $Q'(\zeta) = -\lambda_0(\zeta)^{-1}Q(\zeta)M(\zeta),\,\,\, Q(0) = I$.}. It only affects the boundary conditions slightly, but they remain of the same form. So, without loss of generality, we do not consider the term $M(\zeta)x(\zeta,t)$ in the PDE \eqref{StateEquation_Diag_Uniform}.

Networks of electric lines, chains of density-velocity systems or genetic regulatory networks are modeled by hyperbolic PDEs that fall into the class \eqref{StateEquation_Diag_Uniform}--\eqref{Output_Diag_Uniform}, for more applications, see e.g.~\cite{BastinCoron_Book}. A more general version of that class is presented and analyzed in the book \cite[Chapter 6]{Luo_Guo}. More particularly, the class \eqref{StateEquation_Diag_Uniform}--\eqref{Output_Diag_Uniform} has already been considered in several works, dealing with many different applications. For instance, Lyapunov exponential stability for \eqref{StateEquation_Diag_Uniform}--\eqref{Output_Diag_Uniform} has been studied in \cite{Bastin_2007, Coron_2007_Lyapunov, Diagne_2012} in which the function $\lambda_0$ is replaced by a diagonal matrix with possibly different entries, and in which the boundary conditions, the inputs and the outputs are slightly different compared to \eqref{Input_Diag_Uniform}--\eqref{Output_Diag_Uniform}. Backstepping for boundary control and output feedback stabilization for networks of linear one-dimensional hyperbolic PDEs has been considered in \cite{Auriol_2020, Auriol_BreschPietri, Vazquez_2011_Backstepping}. Boundary feedback control for \eqref{StateEquation_Diag_Uniform}--\eqref{Output_Diag_Uniform} has been studied in \cite{DeHalleux_2003,Prieur_2008_ConservationLaws,Prieur_Winkin_2018}. In these references, the authors consider systems driven by the Saint-Venant-Exner equations as a particular application for illustrating the results. Controllability and finite-time boundary control have also been studied in e.g.~\cite{Chitour_2023, Coron_2021_NullContr} and \cite{Auriol_2016,Coron_2021_Stab_FiniteTime}. Boundary output feedback stabilization for \eqref{StateEquation_Diag_Uniform}--\eqref{Output_Diag_Uniform} has been investigated in \cite{Tanwani2018} in the case of measurement errors, leading to the study of input-to-state stability for those systems. Hyperbolic PDEs similar to \eqref{StateEquation_Diag_Uniform}--\eqref{Output_Diag_Uniform} have attracted attention from a numerical point of view in \cite{Gottlich2017} in which the author proposes a discretization and studies the decay rate of the norm of the trajectories by taking advantage of this discretization. More rencently, Riesz-spectral property of the homogeneous part of \eqref{StateEquation_Diag_Uniform}--\eqref{Output_Diag_Uniform} has been studied in \cite{HastirJacobZwart_Riesz}. In particular, the authors show that the operator dynamics associated to the homogeneous part is a Riesz-spectral operator under the invertibility of some matrices. Moreover, linear-quadratic optimal control has been developed for \eqref{StateEquation_Diag_Uniform}--\eqref{Output_Diag_Uniform} in the case where $d(t)\equiv 0$, see \cite{HastirJacobZwart_LQ}.

The class \eqref{StateEquation_Diag_Uniform}--\eqref{Output_Diag_Uniform} fits also the port-Hamiltonian formalism discussed and analyzed in details in \cite{JacobZwart, Villegas} and \cite{Zwart_ESAIM}. In this case, the analysis of the zero dynamics for the class \eqref{StateEquation_Diag_Uniform}--\eqref{Output_Diag_Uniform} has been studied in \cite{JacobMorrisZwart_zeroDynamics}.

In this work, we derive a solution to the suboptimal $H^\infty$-control problem for the class \eqref{StateEquation_Diag_Uniform}--\eqref{Output_Diag_Uniform}. To this end, we show that \eqref{StateEquation_Diag_Uniform}--\eqref{Output_Diag_Uniform} can be written equivalently as an infinite-dimensional discrete-time system for which the operators are multiplication operators by constant matrices. This has the big advantage of being able to take benefit of the existing theory on suboptimal $H^\infty$-control for finite-dimensional discrete-time systems in order to solve that problem for \eqref{StateEquation_Diag_Uniform}--\eqref{Output_Diag_Uniform}. This approach is particularly useful since it does not need to deal with the unbounded operators that appear because of the boundary inputs and outputs in \eqref{Input_Diag_Uniform}--\eqref{Output_Diag_Uniform}.

The manuscript is organized as follows: Section \ref{SystProp} shows some system properties of \eqref{StateEquation_Diag_Uniform}--\eqref{Output_Diag_Uniform}. The suboptimal $H^\infty$-control problem is described for finite-dimensional discrete-time systems in Section \ref{Hinf_Discrete}. The solution obtained in Section \ref{Hinf_Discrete} is then used to derive the solution of the suboptimal $H^\infty$-control problem for \eqref{StateEquation_Diag_Uniform}--\eqref{Output_Diag_Uniform} in Section \ref{Solution_Conti}. The proposed theory is applied to a vibrating string with velocity measurement and force control in Section \ref{Example}.


\section{System properties}\label{SystProp}

\subsection{Well-posedness analysis}
We start this section by describing the well-posedness of \eqref{StateEquation_Diag_Uniform}--\eqref{Output_Diag_Uniform} as a boundary controlled, boundary observed system. By well-posedness, we mean the following concept: for every $x_0\in X$, every input $u\in\L^2_{\text{loc}}(0,\infty;\mathbb{R}^p)$ and every disturbance $d\in\L^2_{\text{loc}}(0,\infty;\mathbb{R}^k)$, a unique mild solution of \eqref{StateEquation_Diag_Uniform}--\eqref{Input_Diag_Uniform} exists such that the state $x$ and the outputs $y$ and $z$ are in the spaces $X, \L^2_{\text{loc}}(0,\infty;\mathbb{R}^m)$ and $\L^2_{\text{loc}}(0,\infty;\mathbb{R}^l)$, respectively.

Studying well-posedness for \eqref{StateEquation_Diag_Uniform}--\eqref{Output_Diag_Uniform} is eased a lot by considering the system as a boundary controlled port-Hamiltonian system, see e.g. \cite[Chapter 13]{JacobZwart}. Well-posedness is characterized in the next proposition, also available in \cite{HastirJacobZwart_Riesz} or \cite{HastirJacobZwart_LQ}.

\begin{proposition}\label{prop:well-posed}
The boundary controlled boundary observed class of systems \eqref{StateEquation_Diag_Uniform}--\eqref{Output_Diag_Uniform} is well-posed if and only if the matrix $K$ is invertible.
\end{proposition}

\begin{proof}[Proof.]
We only present the main ideas of the proof. We start by defining the operator $A$ by 
\begin{align}
Af &= -\frac{\d}{\d\zeta}(\lambda_0f)\\
D(A) &= \{f\in X, \lambda_0f\in\H^1(0,1;\mathbb{R}^n), K\lambda_0(0)f(0) + L\lambda_0(1)f(1) = 0\}\label{D(A)}.
\end{align}
It is the operator describing the internal dynamics of \eqref{StateEquation_Diag_Uniform}--\eqref{Input_Diag_Uniform}, i.e., the dynamics of \eqref{StateEquation_Diag_Uniform}--\eqref{Input_Diag_Uniform} in which $d \equiv 0$ and $u\equiv 0$. In \cite[Chapter 13]{JacobZwart}, it is shown that \eqref{StateEquation_Diag_Uniform}--\eqref{Output_Reg_Diag_Uniform} is well-posed if and only if the operator $A$ is the generator of a $C_0$-semigroup of bounded linear operators. The proof concludes by noting that this is equivalent to $K$ being invertible, see e.g. \cite[Lemma 1]{HastirJacobZwart_Riesz} or \cite[Proposition 2.1]{HastirJacobZwart_LQ}.
\end{proof}

\subsection{Discrete-time representation}

In this part, we show that the PDEs \eqref{StateEquation_Diag_Uniform} and the boundary inputs and outputs \eqref{Input_Diag_Uniform}--\eqref{Output_Diag_Uniform} may be written as an infinite-dimensional discrete-time system where the operators are multiplication operators by constant matrices. This is performed in a similar way as \cite[Lemma 3.1]{HastirJacobZwart_LQ}. 

\begin{lemma}\label{lem:discrete-representation}
Let us assume that the matrix $K$ is invertible. Moreover, let the functions $k:[0,1]\to [0,1]$ and $p:[0,1]\to [0,\infty)$ be defined as
\begin{align}
k(\zeta) := 1-p(\zeta)p(1)^{-1}, \hspace{0.8cm} p(\zeta) := \int_0^\zeta \lambda_0(\eta)^{-1}\d\eta,
\label{Functions_k_p}
\end{align}
respectively. 
In addition, let the matrices $A_d\in\mathbb{R}^{n\times n}, B_{1,d}\in\mathbb{R}^{n\times k}, B_{2,d}\in\mathbb{R}^{n\times p}, C_{1,d}\in\mathbb{R}^{l\times n}, C_{2,d}\in\mathbb{R}^{m\times n}, D_{11,d}\in\mathbb{R}^{l\times k}, D_{12,d}\in\mathbb{R}^{l\times p}, D_{21,d}\in\mathbb{R}^{m\times k}$ and $D_{22,d}\in\mathbb{R}^{m\times p}$ be respectively given by
\begin{equation}
\begin{array} {lll}
A_d := -K^{-1}L,\hspace{0.5cm} & B_{1,d} := -K^{-1}E, \hspace{0.5cm} & B_{2,d}= -K^{-1}\left[\begin{matrix}0\\ I\end{matrix}\right],\\
C_{1,d} = K_zK^{-1}L-L_z, & C_{2,d} := (K_y K^{-1}L - L_y), & D_{11,d} := K_zK^{-1}E,\\
D_{12,d} = K_zK^{-1}\left[\begin{matrix}0\\ I\end{matrix}\right], & D_{21,d} = K_yK^{-1}E, & D_{22,d} = K_yK^{-1}\left[\begin{matrix}0\\ I\end{matrix}\right].\label{Matrices_Discrete}
\end{array}
\end{equation}
The continuous-time system \eqref{StateEquation_Diag_Uniform}--\eqref{Output_Diag_Uniform} may be equivalently written as the following discrete-time system
\begin{align}
x_d(j+1)(\zeta) &= A_dx_d(j)(\zeta) + B_{1,d}d_d(j) + B_{2,d} u_d(j)(\zeta),\label{State_Discrete}\\
x_d(0)(k(\zeta)) &= \lambda_0(\zeta)x_0(\zeta),\label{Init_Discrete}\\
z_d(j)(\zeta) &= C_{1,d}x_d(j)(\zeta) + D_{11,d}d_d(j)(\zeta) + D_{12,d} u_d(j)(\zeta),\label{Output_Reg_Discrete}\\
y_d(j)(\zeta) &= C_{2,d}x_d(j)(\zeta) + D_{21,d} d_d(j)(\zeta) + D_{22,d} u_d(j)(\zeta).\label{Output_Discrete}
\end{align}
In \eqref{State_Discrete}--\eqref{Output_Discrete} $j\in\mathbb{N}, \zeta\in [0,1], x_d(j)\in X, d_d(j)\in\mathcal{W}, u_d(j)\in \mathcal{U}, z_d(j)\in\mathcal{Z}$ and $y_d(j)\in\mathcal{Y}$ where
\begin{align*}
\mathcal{W} := \L^2(0,1;\mathbb{R}^k), \hspace{0.1cm} \mathcal{U} := \L^2(0,1;\mathbb{R}^p), \hspace{0.1cm} \mathcal{Z} := \L^2(0,1;\mathbb{R}^l),\hspace{0.1cm} \mathcal{Y} := \L^2(0,1;\mathbb{R}^m).
\end{align*}
In addition, the functions $x_d, d_d, u_d, z_d$ and $y_d$ are related to the state, the input, the disturbance and the outputs of the continuous-time system \eqref{StateEquation_Diag_Uniform}--\eqref{Output_Diag_Uniform} via the following relations
\begin{align*}
x_d(j)(\zeta) &= f(j+\zeta),\hspace{0.8cm} j\geq 1,\\
d_d(j)(\zeta) &= d((j+\zeta)p(1)), \hspace{0.8cm} j\in\mathbb{N},\\
u_d(j)(\zeta) &= u((j+\zeta)p(1)),\hspace{0.8cm} j\in\mathbb{N},\\
z_d(j)(\zeta) &= z((j+\zeta)p(1)), \hspace{0.8cm} j\in\mathbb{N},\\
y_d(j)(\zeta) &= y((j+\zeta)p(1)),\hspace{0.8cm} j\in\mathbb{N},
\end{align*}
where $f\left(k(\zeta) + p(1)^{-1}t\right) = \lambda_0(\zeta)x(\zeta,t)$ for $\zeta\in [0,1]$ and $t\geq 0$.
\end{lemma}

\begin{proof}[Proof.]
Let $p$ and $k$ be defined as in \eqref{Functions_k_p}. Since $\lambda_0>0$, $p$ is a monotonic function satisfying $p(0) = 0$. In addition, the function $k$ has the properties that $k(0) = 1$ and $k(1) = 0$.
Observe that the solution of \eqref{StateEquation_Diag_Uniform} is given by $x(\zeta,t) = \lambda_0(\zeta)^{-1}f\left(k(\zeta) + p(1)^{-1}t\right)$ for some function $f$. Remark that, according to the definition of $p$, there holds $p(\zeta)p(1)^{-1}\in [0,1]$ for every $\zeta\in [0,1]$. By substituting the expression of $x$ into the initial conditions, the boundary conditions and the output equations \eqref{StateEquation_Diag_Uniform}--\eqref{Output_Diag_Uniform}, we get
\begin{align*}
f(k(\zeta)) &= \lambda_0(\zeta)x_0(\zeta),\hspace{0.5cm} \zeta\in [0,1],\\
Ed(t) + \left[\begin{matrix}0\\ I\end{matrix}\right]u(t) &= -K f(1+p(1)^{-1}t) - L f(p(1)^{-1}t),\\
z(t) &= -K_z f(1+p(1)^{-1}t) - L_z f(p(1)^{-1}t),\\
y(t) &= -K_y f(1+p(1)^{-1}t) - L_y f(p(1)^{-1}t).
\end{align*}
Using the invertibility of $K$, we find
\begin{align}
f(1+p(1)^{-1}t) &= -K^{-1}Lf(p(1)^{-1}t) - K^{-1} \left[\begin{matrix}0\\ I\end{matrix}\right]u(t) - K^{-1}Ed(t),\label{State_f}\\
z(t) &= (K_zK^{-1}L-L_z) f(p(1)^{-1}t) + K_zK^{-1}Ed(t) + K_zK^{-1}\left[\begin{matrix}0\\ I\end{matrix}\right]u(t),\label{Output_Reg_f}\\
y(t) &= (K_y K^{-1}L-L_y)f(p(1)^{-1}t) + K_yk^{-1}Ed(t) + K_y K^{-1}\left[\begin{matrix}0\\ I\end{matrix}\right]u(t)\label{Output_f}.
\end{align}
By defining $A_d\in\mathbb{R}^{n\times n}, B_{1,d}\in\mathbb{R}^{n\times k} B_{2,d}\in\mathbb{R}^{n\times p}, C_{1,d}\in\mathbb{R}^{l\times n}, C_{2,d}\in\mathbb{R}^{m\times n}, D_{11,d}\in\mathbb{R}^{l\times k}, D_{12,d}\in\mathbb{R}^{l\times p}, D_{21,d}\in\mathbb{R}^{m\times k}$ and $D_{2,d}\in\mathbb{R}^{m\times p}$ as in \eqref{Matrices_Discrete}, we have that \eqref{State_f}--\eqref{Output_f} may be written as
\begin{align}
f(1+p(1)^{-1}t) &= A_d f(p(1)^{-1}t) + B_{1,d}d(t) + B_{2,d} u(t),\label{State_Equiv_f}\\
z(t) &= C_{1,d} f(p(1)^{-1}t) + D_{11,d}d(t) + D_{12,d}u(t),\label{Output_Reg_Equiv_f}\\
y(t) &= C_{2,d} f(p(1)^{-1}t) + D_{21,d}d(t) + D_{22,d} u(t).\label{Output_Equiv_f}
\end{align}
Next for $j\in\mathbb{N}$ and $\zeta\in [0,1]$ we define
\begin{align*}
&x_d(j)\in\L^2(0,1;\mathbb{R}^n),\hspace{0.1cm} d_d(j)\in\L^2(0,1;\mathbb{R}^k), \hspace{0.1cm} u_d(j)\in\L^2(0,1;\mathbb{R}^p),\\
&z_d(j)\in\L^2(0,1;\mathbb{R}^l),\hspace{0.1cm} y_d(j)\in\L^2(0,1;\mathbb{R}^m),
\end{align*}
by 
\begin{align*}
x_d(0)(k(\zeta)) &:= \lambda_0(\zeta)z_0(\zeta),\\
x_d(j)(\zeta) &:= f(j+\zeta),\hspace{0.5cm} j\geq 1,\\
d_d(j)(\zeta) &:= d((j+\zeta)p(1)),\hspace{0.5cm} j\in\mathbb{N},\\
u_d(j)(\zeta) &:= u((j+\zeta)p(1)),\hspace{0.5cm} j\in\mathbb{N},\\
z_d(j)(\zeta) &:= z((j+\zeta)p(1)),\hspace{0.5cm} j\in\mathbb{N},\\
y_d(j)(\zeta) &:= y((j+\zeta)p(1)),\hspace{0.5cm} j\in\mathbb{N}.
\end{align*}
Remark that, for any $t\geq 0$, we can find a unique $j\in\mathbb{N}$ and $\zeta\in [0,1]$ such that $j+\zeta = p(1)^{-1}t$. With this definition, we get 
\begin{align*}
f(p(1)^{-1}t) &= f(j+\zeta) = x_d(j)(\zeta),\\
d(t) &= d((j+\zeta)p(1)) = d_d(j)(\zeta),\\
u(t) &= u((j+\zeta)p(1)) = u_d(j)(\zeta),\\
z(t) &= z((j+\zeta)p(1)) = z_d(j)(\zeta),\\
y(t) &= y_d(j)(\zeta),
\end{align*}
which allows us to write equivalently \eqref{State_Equiv_f}, \eqref{Output_Reg_Equiv_f} and \eqref{Output_Equiv_f} as \eqref{State_Discrete}, \eqref{Output_Reg_Discrete} and \eqref{Output_Discrete}, respectively.
\end{proof}

\begin{remark}
The discrete-time representation \eqref{State_Discrete}--\eqref{Output_Discrete} is useful for studying system theoretic properties of \eqref{StateEquation_Diag_Uniform}--\eqref{Output_Diag_Uniform} such as internal stability of the dynamics. In particular, the $C_0$-semigroup generated by the operator $A$ is exponentially stable if and only if the matrix $A_d$ has its spectral radius less than 1, see e.g. \cite{HastirJacobZwart_Riesz} or \cite{JacobMorrisZwart_zeroDynamics}. 
\end{remark}


\section{\texorpdfstring{Suboptimal $H^\infty$}{}-control for discrete-time finite-dimensional systems}\label{Hinf_Discrete}

In this section, we show how the solution of the suboptimal $H^\infty$-control problem can be computed for finite-dimensional discrete-time systems. Let $\Sigma$ and $\Sigma_c$ be the systems in Figure \ref{fig:Diagram_Closed_Loop} and let us denote the closed-loop system by $\Sigma_{cl}$.

We assume that $\Sigma$, see Figure \ref{fig:Diagram_Open_Loop}, has the following structure
\begin{equation}
  \Sigma:\left\{\begin{array}{l}
    x(j+1) = A_d x(j) + B_{1,d}d(j) + B_{2,d} u(j)\\
    z(j) = C_{1,d} x(j) + D_{11,d}d(j) + D_{12,d} u(j)\\
    y(j) = C_{2,d} x(j) + D_{21,d}d(j) + D_{22,d} u(j),  
  \end{array}\right.
  \label{Sigma_Block}
\end{equation}
where, for $j\in\mathbb{N}$, $x(j)\in\mathbb{R}^n, d(j)\in\mathbb{R}^k, u(j)\in\mathbb{R}^p, z(j)\in\mathbb{R}^l$ and $y(j)\in\mathbb{R}^m$. $A_d, B_{i,d}, C_{i,d}$ and $D_{\alpha\beta,d}, i=1,2, \alpha = 1, 2, \beta=1, 2$ are matrices of appropriate dimensions. We emphasize that, despite we used the same notations as in Section \ref{SystProp}, the matrices depicted in \eqref{Sigma_Block} are not necessarily the same. We kept the same notation because it will be useful later in the manuscript. Following the lines of \cite[Chapter 10]{IonescuWeiss_99}, the dynamic controller $\Sigma_c$, as depicted in Figure \ref{fig:Diagram_Closed_Loop}, is assumed to be of the form
\begin{equation}
  \Sigma_c:\left\{\begin{array}{l}
    x_c(j+1) = A_c x_c(j) + B_c u_c(j)\\
    y_c(j) = C_c x_c(j) + D_c u_c(j),  
  \end{array}\right.
  \label{Sigma_c_Block}
\end{equation}
where $A_c, B_c, C_c$ and $D_c$ are matrices that need to be designed in order for the submoptimal $H^\infty$-control problem to be solved. In the controller $\Sigma_c$, it is assumed that $u_c(j)\in\mathbb{R}^m$ and $y_c(j)\in\mathbb{R}^p$. The vector $x_c(j)$ could be of any size. For this reason, we do not specify it. The interconnection between $\Sigma$ and $\Sigma_c$ is obtained by setting $u_c = y$ and $y_c = u$. This gives
\begin{equation}
  \begin{array}{l}
    \left(\begin{matrix}x(j+1)\\x_c(j+1)\end{matrix}\right) = \left(\begin{matrix}A_d & 0\\ 0 & A_c\end{matrix}\right) \left(\begin{matrix}x(j)\\x_c(j)\end{matrix}\right) + \left(\begin{matrix}B_{1,d}d(j) + B_{2,d} u(j)\\ B_c y(j)\end{matrix}\right)\\
    z(j) = C_{1,d}x(j) + D_{11,d}d(j) + D_{12,d}u(j)\\
    y(j) = C_{2,d}x(j) + D_{21,d}d(j) + D_{22,d}u(j)\\
    u(j) = C_c x_c(j) + D_c y(j).
  \end{array}
  \label{prelim_closed-loop}
\end{equation}
As the closed-loop system has input $d$ and output $z$, we need to eliminate the variables $u(j)$ and $y(j)$ from the above equations. Combining the two last equations of \eqref{prelim_closed-loop} gives
\begin{equation*}
\left(\begin{matrix}I & -D_c\\-D_{22,d} & I\end{matrix}\right)\left(\begin{matrix}u(j)\\ y(j)\end{matrix}\right) = \left(\begin{matrix}0 & C_c\\ C_{2,d} & 0\end{matrix}\right)\left(\begin{matrix}x(j)\\ x_c(j)\end{matrix}\right) + \left(\begin{matrix}0\\ D_{21,d}\end{matrix}\right)d(j),
\end{equation*}
which shows that the invertibility of the matrix $I-D_{22,d}D_c$ (equivalently the matrix $I-D_cD_{22,d}$) is an equivalent condition for the closed-loop system to be well-posed. We shall make this assumption from now on. 
\begin{proposition}\label{rem:invertibility_S}
The invertibility of $\left(\begin{smallmatrix}I & -D_c\\-D_{22,d} & I\end{smallmatrix}\right)$ is equivalent to the invertibility of $I-D_{22,d}D_c$ (equivalently to the invertibility of $I-D_cD_{22,d}$).
\end{proposition}
\begin{proof}[Proof.]
The proof follows by noting that
\begin{align*}
  \mathrm{det}\left(\begin{smallmatrix}I & -D_c\\-D_{22,d} & I\end{smallmatrix}\right) = \mathrm{det}(I-D_{22,d}D_c)
\end{align*}
and 
\begin{align*}
  \mathrm{det}\left(\begin{smallmatrix}I & -D_c\\-D_{22,d} & I\end{smallmatrix}\right) = \mathrm{det}(I-D_cD_{22,d}).
\end{align*}
\end{proof} 
Note that, in the case where $I - D_cD_{22,d}$ (equivalently, $I-D_{22,d}D_c$) is invertible, there holds 
\begin{equation*}
(I - D_cD_{22,d})^{-1} = I + D_c(I-D_{22,d}D_c)^{-1}D_{22,d}.
\end{equation*}

We shall denote by $S$ and $\tilde{S}$ the matrices $I-D_{22,d}D_c$ and $I-D_cD_{22,d}$, respectively. Taking Proposition \ref{rem:invertibility_S} into account, the dynamics of the closed-loop may be written as 
\begin{equation}
  \Sigma_{cl}:\left\{\begin{array}{l}
    \left(\begin{matrix}x(j+1)\\x_c(j+1)\end{matrix}\right) =  A_{cl}\left(\begin{matrix}x(j)\\x_c(j)\end{matrix}\right) + B_{cl}d(i)\\
    z(j) = C_{cl}\left(\begin{matrix}x(j)\\x_c(j)\end{matrix}\right) + D_{cl}d(j),
  \end{array}\right.
  \label{Sigma_closed-loop}
\end{equation}
where the matrices $A_{cl}, B_{cl}, C_{cl}$ and $D_{cl}$ are given by 
\begin{align}
A_{cl} &= \left(\begin{matrix}A_d + B_{2,d}\tilde{S}^{-1}D_cC_{2,d} & B_{2,d}\tilde{S}^{-1}C_c\\ B_cS^{-1}C_{2,d} & A_c+B_cS^{-1}D_{22,d}C_c\end{matrix}\right),\label{A_cl}\\
B_{cl} &= \left(\begin{matrix}B_{1,d} + B_{2,d}\tilde{S}^{-1}D_cD_{21,d}\\ B_c S^{-1}D_{21,d}\end{matrix}\right),\label{B_cl}\\
C_{cl} &= \left(\begin{matrix}C_{1,d}+D_{12,d}D_cS^{-1}C_{2,d} & D_{12,d}\tilde{S}^{-1}C_c\end{matrix}\right),\label{C_cl}\\
D_{cl} &= (D_{11,d} + D_{12,d}D_cS^{-1}D_{21,d}),\label{D_cl}
\end{align}
respectively. In the langage of \cite[Chapter 2]{IonescuWeiss_99}, it is said that the closed-loop system as presented in \eqref{Sigma_closed-loop} is the \textit{linear fractional transormation} (short LFT) of $\Sigma$ with $\Sigma_c$. We shall use this concept later in the main theorem of this section.

In what follows, the notation $\mathbb{D}$ stands for the open unit disc of the complex plane and ``$\complement$" denotes the complement of a set.

The suboptimal $H^\infty$-control problem for $\Sigma$ is defined as follows.
\begin{definition}\label{Def:Hinf_SubOpt}
  Given a positive parameter $\gamma$, the suboptimal $H^\infty$-control problem for $\Sigma$ is solved if and only if the three following conditions are satisfied:
  \begin{itemize}
    \item The closed-loop system is well-defined, that is, the matrix $\left(\begin{smallmatrix}I & -D_c\\-D_{22,d} & I\end{smallmatrix}\right)$ is invertible;
    \item $\Sigma_{cl}$ is internally stable, that is, the matrix $A_{cl}$ is stable in the sense that $r(A_{cl})<1$ where $r$ is the spectral radius;
    \item The transfer function $\mathbf{G}_{cl}$ of $\Sigma_{cl}$ that goes from $d$ to $z$ satisfies $\mathbf{G}_{cl}\in H^\infty(\overline{\mathbb{D}}^\complement,\mathbb{C}^{m\times p})$ and has $H^\infty$-norm less than $\gamma$.
  \end{itemize}
  \end{definition}
  
  The Hardy space $H^\infty(\overline{\mathbb{D}}^\complement,\mathbb{C}^{m\times p})$ consists of functions $\mathbf{G}:\overline{\mathbb{D}}^\complement\to\mathbb{C}^{m\times p}$ whose elements are analytic on $\overline{\mathbb{D}}^\complement$ and that satisfy $\sup_{s\in\overline{\mathbb{D}}^\complement} \Vert \mathbf{G}(s)\Vert < \infty$. The $H^\infty$-norm of $\mathbf{G}\in H^\infty(\overline{\mathbb{D}}^\complement,\mathbb{C}^{m\times p})$ is defined as
  \begin{equation*}
    \Vert \mathbf{G}\Vert_{\infty,\overline{\mathbb{D}}^\complement} := \sup_{\theta\in [0,2\pi]}\Vert \mathbf{G}(e^{i\theta})\Vert,
  \label{Hinf_norm_Discr_Finite}  
  \end{equation*}
  Note that the shortcut $H^\infty(\overline{\mathbb{D}}^\complement)$ will be used to avoid heavy notations.

\subsection{Assumptions}
We give here the assumptions that are needed in order to solve the suboptimal $H^\infty$-control problem, see \cite[Section 10.11]{IonescuWeiss_99}. First we comment on two basic assumptions that are made in \cite[Chapter 10]{IonescuWeiss_99}.

\begin{remark}\label{rem:D_22_nonzero}
\begin{itemize} 
\item It is assumed in \cite[Chapter 10]{IonescuWeiss_99} that $\gamma = 1$ holds without loss of generality. Indeed, as explained in \cite[Chapter 10]{IonescuWeiss_99}, if $\gamma\neq 1$, the original system can always be rescaled in such a way that the solution to the suboptimal $H^\infty$-control problem for the rescaled system with $\gamma=1$ gives the solution to the suboptimal $H^\infty$-control problem for the original system with the original $\gamma$. The matrices of the rescaled system are defined as
\begin{equation}
\begin{array}{lll}
A_d^{scaled} := A_d, & B_{1,d}^{scaled} := \gamma^{-\frac{1}{2}}B_{1,d}, & B_{2,d}^{scaled} := \gamma^\frac{1}{2}B_{2,d},\\
C_{1,d}^{scaled} := \gamma^{-\frac{1}{2}}C_{1,d}, & C_{2,d}^{scaled} := \gamma^{\frac{1}{2}}C_{2,d}, & D_{11,d}^{scaled} := \gamma^{-1}D_{11,d},\\
D_{12,d}^{scaled} := D_{12,d}, & D_{21,d}^{scaled} := D_{21,d}, & D_{22,d}^{scaled} := \gamma D_{22,d}.
\end{array}
\label{Matrices_Scaled}
\end{equation}
The solution to the suboptimal $H^\infty$-control problem for \eqref{Matrices_Scaled} with $\gamma = 1$ gives rise to a controller $\Sigma_c$ whose transfer function is given by $\mathbf{G}_c^{scaled} := \gamma^{-1}\mathbf{G}_c$, where $\mathbf{G}_c$ is the transfer function that would have been obtained by solving the suboptimal $H^\infty$-control problem for the original system. Moreover, the closed-loop transfer function for the rescaled system is given by $\mathbf{G}_{cl}^{scaled} := \gamma^{-1}\mathbf{G}_{cl}$ and it has $H^\infty$-norm less than $1$, which implies that $\mathbf{G}_{cl}$ has $H^\infty$-norm less than $\gamma$, as desired.
\item In \cite[Section 10.11]{IonescuWeiss_99} is the assumption that the matrix $D_{22,d}$ is the null matrix. It is mentioned that this assumption can be made without loss of generality. Indeed, in the case where $D_{22,d}\neq 0$, the suboptimal $H^\infty$-control problem has to be solved in the same way as if $D_{22,d} = 0$. This gives rise to a controller $\Sigma_c$ whose transfer function is denoted by $\mathbf{G}_c$. Then, the controller for the case with $D_{22,d}\neq 0$ is given by $\mathbf{G}_c(I + D_{22,d}\mathbf{G}_c)^{-1}$.
\end{itemize}
\end{remark}

The following assumption is needed to ensure the existence of stabilizing controllers.

\begin{assumption}\label{Assum:Stab_Detec}
The pair $(A_d,B_{2,d})$ is stabilizable and the pair $(C_{2,d},A_d)$ is detectable.
\end{assumption}

Now we make the two following regularity assumptions, see \cite[Section 10.11]{IonescuWeiss_99}.

\begin{assumption}\label{Assum:Regularity}
  For every $\theta\in [0,2\pi]$, the two following rank conditions are satisfied:
  \begin{itemize}
    \item[1.] \begin{equation*}
      \rank \left(\begin{matrix}e^{i\theta}-A_d & -B_{2,d}\\
      -C_{1,d} & -D_{12,d}\end{matrix}\right) = n + p.
    \end{equation*}
    \item[2.] \begin{equation*}
    \rank\left(\begin{matrix}e^{i\theta}-A_d & -B_{1,d}\\
      -C_{2,d} & -D_{21,d}\end{matrix}\right) = n + k.
    \end{equation*}
  \end{itemize}
  This has the consequence that there are more regulated outputs than controlled inputs and that the number of external inputs is larger than the number of measured outputs.
\end{assumption}

\subsection{The solution}

Before presenting the solution, we need to introduce the following concept, see \cite[Chapter 3 and Section 10.12]{IonescuWeiss_99}.

\begin{definition}\label{Def:KSPYS}
Let $\Sigma_* = (A_d, B_d, Q, L, R)$ with $A_d\in\mathbb{R}^{n\times n}, B_d\in\mathbb{R}^{n\times m}, Q\in\mathbb{R}^{n\times n}, L\in\mathbb{R}^{n\times m}, R\in\mathbb{R}^{m\times m}$ and $Q = Q^T, R = R^T$, and let $J = \left(\begin{smallmatrix}-I_{m_1} & 0\\0 & I_{m_2}\end{smallmatrix}\right)$ with $m_1+m_2 = m$. The \textit{Kalman-Szego-Popov-Yakubovich} system associated with $\Sigma_*$ and $J$ has a solution $(X,V,W)$ with $X\in\mathbb{R}^{n\times n}, V\in\mathbb{R}^{m\times m}$ and $W\in\mathbb{R}^{m\times n}$ if
\begin{align}
R + B_d^TXB_d &= V^TJV\label{EqV_KSPYS}\\
L + A_d^TXB_d &= W^TJV\label{EqWV_KSPYS}\\
Q + A_d^TXA_d-X &= W^TJW\label{EqW_KSPYS},
\end{align}
with $R + B_d^TXB_d$ nonsingular. Moreover, the solution is called stabilizing if the matrix
\begin{equation*}
A_d - B_d\left(R + B_d^TXB_d\right)^{-1}\left(L^T + B_d^TXA_d\right)
\end{equation*}
is stable. 
\end{definition}

\begin{remark}
According to \cite[Remark 3.6.9]{IonescuWeiss_99}, if $(X = X^T, V, W)$ is a solution to \eqref{EqV_KSPYS}--\eqref{EqW_KSPYS}, then $X$ is a symmetric solution to the following algebraic Riccati equation 
\begin{equation}
A_d^TXA_d - X + Q = (L + A_d^TXB_d)(R + B_d^TXB_d)^{-1}(L^T + B_d^TXA_d).\label{Riccati_KSPYS}
\end{equation}
This can be seen by eliminating the matrix $V$ in \eqref{EqV_KSPYS}--\eqref{EqWV_KSPYS}, by using the fact that $R + B_d^TXB_d$ is nonsingular. In addition, if $X$ is a symmetric solution to \eqref{Riccati_KSPYS}, then we can choose a $V$ such that \eqref{EqV_KSPYS} is satisfied, provided that $\sign(R + B_d^TXB_d) = J$. Then, a matrix $W$ can be defined such that \eqref{EqWV_KSPYS} is satisfied. The so-defined matrix $W$ will necessarily satisfy \eqref{EqW_KSPYS}. In summary, the Kalman-Szego-Popov-Yakubovich system associated with $\Sigma_*$ and $J$ has a solution $(X = X^T, V, W)$ if and only if $X$ satisfies the algebraic Riccati equation \eqref{Riccati_KSPYS} with $\sign(R + B_d^TXB_d) = J$. Note that, for a given Hermitian matrix $P$, $\sign(P)$ is defined as 
\begin{equation*}
\sign(P) = \left(\begin{matrix}-I_{n_-} & 0 & 0\\
0 & I_{n_+} & 0\\
0 & 0 & O_{n_0}\end{matrix}\right),
\end{equation*}
with $n_-, n_+$ and $n_0$ being the number of negative, positive and zero eignevalues of $P$, respectively.
\end{remark}

The following theorem gives a characterization of the solution of the suboptimal $H^\infty$-control problem for $\Sigma$ \eqref{Sigma_Block} with $\gamma = 1$, see \cite[Theorem 10.12.1]{IonescuWeiss_99}.

\begin{theorem}\label{Thm_Hinf_Discr_Finite}
Let $\Sigma_* := (A_d,\left(\begin{matrix}B_{1,d} & B_{2,d}\end{matrix}\right), Q_c, L_c, R_c)$ and $\Sigma_o := (A_d^T,\left(\begin{matrix}C_{1,d}^T & C_{2,d}^T\end{matrix}\right), Q_o, L_o, R_o)$, where 
\begin{equation*}
\begin{array}{lll}
Q_c = C_{1,d}^TC_{1,d}, & L_c = C_{1,d}^T\left(\begin{smallmatrix}D_{11,d} & D_{12,d}\end{smallmatrix}\right), & R_c = \left(\begin{smallmatrix}D_{11,d}^T\\ D_{12,d}^T\end{smallmatrix}\right)\left(\begin{smallmatrix}D_{11,d} & D_{12,d}\end{smallmatrix}\right) - \left(\begin{smallmatrix}I_l & 0\\ 0 & 0\end{smallmatrix}\right),\\
Q_o = B_{1,d}B_{1,d}^T, & L_o = B_{1,d}\left(\begin{smallmatrix}D_{11,d}^T & D_{21,d}^T\end{smallmatrix}\right), & R_o = \left(\begin{smallmatrix}D_{11,d}\\ D_{21,d}\end{smallmatrix}\right)\left(\begin{smallmatrix}D_{11,d}^T & D_{21,d}^T\end{smallmatrix}\right) - \left(\begin{smallmatrix}I_k & 0\\ 0 & 0\end{smallmatrix}\right).
\end{array}
\end{equation*}
The suboptimal $H^\infty$-control problem formulated in Definition \ref{Def:Hinf_SubOpt} has a solution with $\gamma=1$ if and only if the following conditions are satisfied
\begin{itemize}
\item[a.] The Kalman-Szego-Popov-Yakubovich system associated with $\Sigma_*$ and $J_* = \left(\begin{smallmatrix}-I_l & 0\\0 & I_m\end{smallmatrix}\right)$ has a stabilizing solution $(X,V_c,W_c)$ with\footnote{The splitting in $V_c$ and $W_c$ is such that $V_{c,11}\in\mathbb{R}^{l\times l}, V_{c,21}\in\mathbb{R}^{m\times l}, V_{c,22}\in\mathbb{R}^{m\times m}$ and $W_{c,1}\in\mathbb{R}^{l\times n}, W_{c,2}\in\mathbb{R}^{m\times n}$.} $V_c = \left(\begin{smallmatrix}V_{c,11} & 0\\ V_{c,21} & V_{c,22}\end{smallmatrix}\right)$ and $W_c = \left(\begin{smallmatrix}W_{c,1}\\ W_{c,2}\end{smallmatrix}\right)$ with $X\geq 0$;
\item[b.] The Kalman-Szego-Popov-Yakubovich system associated with $\Sigma_o$ and $J_o = \left(\begin{smallmatrix}-I_k & 0\\0 & I_p\end{smallmatrix}\right)$ has a stabilizing solution $(Y,V_o,W_o)$ with\footnote{The splitting in $V_o$ and $W_o$ is such that $V_{o,11}\in\mathbb{R}^{k\times k}, V_{o,21}\in\mathbb{R}^{p\times k}, V_{o,22}\in\mathbb{R}^{p\times p}$ and $W_{o,1}\in\mathbb{R}^{k\times n}, W_{o,2}\in\mathbb{R}^{p\times n}$.} $V_o = \left(\begin{smallmatrix}V_{o,11} & 0\\ V_{o,21} & V_{o,22}\end{smallmatrix}\right)$ and $W_o = \left(\begin{smallmatrix}W_{o,1}\\ W_{o,2}\end{smallmatrix}\right)$ with $Y\geq 0$;
\item[c.] The spectral radius $\rho(XY)$ is strictly less than $1$.
\end{itemize}
Assume that part a. is satisfied, let $J_\times = \left(\begin{smallmatrix}-I_m & 0\\ 0 & I_p\end{smallmatrix}\right)$ and consider 
\begin{align*}
  \Sigma_\times = \left(A_d^T + F_1^TB_{1,d}^T, \left(\begin{matrix}-F_2^TV_{c,22}^T & C_{2,d}^T + F_1^TD_{21,d}^T\end{matrix}\right),Q_\times, L_\times, R_\times\right),
\end{align*}
where $F = \left(\begin{smallmatrix}F_1\\ F_2\end{smallmatrix}\right)$ is the stabilizing feedback solution to the Kalman-Szego-Popov-Yakubovich system associated to $\Sigma_*$ and $J_*$ and
\begin{align*}
Q_\times &= B_{1,d}\left(V_{c,11}^TV_{c,11}\right)^{-1}(B_{1,d})^T,\\
L_\times &= B_{1,d}\left(V_{c,11}^TV_{c,11}\right)^{-1}\left(\begin{smallmatrix}V_{c,21}^T & D_{21,d}^T\end{smallmatrix}\right),\\
R_\times &= \left(\begin{smallmatrix}V_{c,21}\\ D_{21,d}\end{smallmatrix}\right)\left(V_{c,11}^TV_{c,11}\right)^{-1}\left(\begin{smallmatrix}V_{c,21}^T & D_{21,d}^T\end{smallmatrix}\right) - \left(\begin{smallmatrix}I_m & 0\\0 & 0\end{smallmatrix}\right).
\end{align*}
Now assume that parts a., b. and c. are satisfied. Then, the Kalman-Szego-Popov-Yakubovich system associated to $\Sigma_\times$ and $J_\times$ has a stabilizing solution $(Z,V_\times,W_\times)$ with\footnote{The splitting in $V_\times$ and $W_\times$ is such that $V_{\times,11}\in\mathbb{R}^{m\times m}, V_{\times,21}\in\mathbb{R}^{p\times m}, V_{\times,22}\in\mathbb{R}^{p\times p}$ and $W_{\times,1}\in\mathbb{R}^{m\times n}, W_{\times,2}\in\mathbb{R}^{p\times n}$.}  $V_\times = \left(\begin{smallmatrix}V_{\times,11} & 0\\ V_{\times,21} & V_{\times,22}\end{smallmatrix}\right), W_\times = \left(\begin{smallmatrix}W_{\times,1}\\ W_{\times,2}\end{smallmatrix}\right)$ and $Z = Y(I-XY)^{-1}\geq 0$. Let 
\begin{equation*}
\begin{array}{lll}
C_{2,F_1} &:= C_{2,d} + D_{21,d}F_1, & S_c := (V_{c,11}^TV_{c,11})^{-1},
S_\times := (V_{\times,11}^TV_{\times,11})^{-1}.
\end{array}
\end{equation*}
Then, the class of all solutions to the suboptimal $H^\infty$-control problem stated in Definition \ref{Def:Hinf_SubOpt} with $\gamma = 1$ can be expressed as 
\begin{align*}
\Sigma_c = (A_c,B_c,C_c,D_c) = \mathrm{LFT}(\Sigma_g,\Sigma_Q),
\end{align*}
where $\Sigma_Q = (A_Q,B_Q,C_Q,D_Q)$ is a discrete-time system, whose transfer function satisfies $\mathbf{G}_{Q}\in H^\infty(\overline{\mathbb{D}}^\complement)$ and $\Vert \mathbf{G}_Q\Vert_{\infty,\overline{\mathbb{D}}^\complement} < 1$. Here $\Sigma_g = \left(A_g,\left(\begin{smallmatrix}B_{g,1} & B_{g,2}\end{smallmatrix}\right),\left(\begin{smallmatrix}C_{g,1}\\C_{g,2}\end{smallmatrix}\right),\left(\begin{smallmatrix}D_{g,11} & D_{g,12}\\ D_{g,21} & D_{g,22}\end{smallmatrix}\right)\right)$ is a discrete-time system of the form \eqref{Sigma_c_Block}, with 
\begin{align*}
A_g &= A_d + B_{1,d}F_1 + B_{2,d}F_2 + B_{g,1}C_{2,F_1},\\
B_{g,1} &= -B_{1,d}S_cD_{21,d}^T\left(D_{21,d}S_cD_{21,d}^T + C_{2,F_1}ZC_{2,F_1}^T\right)^{-1} - B_{2,d}D_{g,11} - (A_d + B_{1,d}F_1)ZC_{g,2}^TD_{g,21},\\
B_{g,2} &= -B_{2,d}D_{g,12} - B_{1,d}S_c\left(V_{c,22}^{-1}V_{c,21} + D_{g,11}D_{21,d}\right)^TV_{c,22}^TS_\times^\frac{1}{2} - (A_d + B_{1,d}F_1)ZC_{g,1}^TV_{c,22}^TS_\times^\frac{1}{2},\\
C_{g,1} &= -F_2 + D_{g,11}C_{2,F_1},\\
C_{g,2} &= \left(D_{21,d}S_cD_{21,d}^T + C_{2,F_1}ZC_{2,F_1}^T\right)^{-\frac{1}{2}}C_{2,F_1},\\
D_{g,11} &= -\left(V_{c,22}^{-1}V_{c,21}S_cD_{21,d}^T - F_2ZC_{2,F_1}^T\right)\left(D_{21,d}S_cD_{21,d}^T + C_{2,F_1}ZC_{2,F_1}^T\right)^{-1},\\
D_{g,12} &= V_{c,22}^{-1}S_\times^{-\frac{1}{2}},\\
D_{g,21} &= \left(D_{21,d}S_cD_{21,d}^T + C_{2,F_1}ZC_{2,F_1}^T\right)^{-\frac{1}{2}},\\
D_{g,22} &= 0.
\end{align*}
\end{theorem}

We will not give a proof of that theorem. Interested readers are referred to \cite[Theorem 10.12.1]{IonescuWeiss_99}. Nevertheless, the closed-loop system as the linear fractional of $\Sigma$ and $\Sigma_c$, where $\Sigma_c$ is the linear fractional transformation of $\Sigma_g$ and $\Sigma_Q$, is depicted in Figure \ref{fig:LFT_LFT_Closed-Loop}. 

\begin{figure}
  \centering
  \includegraphics[scale=1.5]{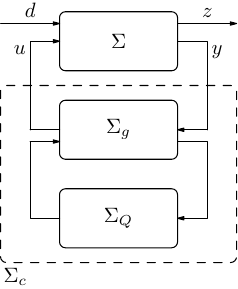}
  \caption{Closed-loop system resulting from the interconnection of $\Sigma$ and $\Sigma_c$ with $\Sigma_c$ being the \textit{left fractional transformation} of $\Sigma_g$ and $\Sigma_Q$.\label{fig:LFT_LFT_Closed-Loop}}
\end{figure}

\begin{remark}\label{rem:ScaledSystem_Sigma_g}
\begin{itemize} 
  \item The system $\Sigma_g^{scaled}$ obtained by applying Theorem \ref{Thm_Hinf_Discr_Finite} to the scaled system \eqref{Matrices_Scaled} can be expressed in terms of the system $\Sigma_g$ obtained for the original system $\Sigma$ in the following way
\begin{equation}
  \begin{array}{lll}
  A_g^{scaled} := A_g, & B_{g,1}^{scaled} := \gamma^{-\frac{1}{2}}B_{g,1}, & B_{g_2}^{scaled} := \gamma^\frac{1}{2}B_{g,2},\\
  C_{g,1}^{scaled} := \gamma^{-\frac{1}{2}}C_{g,1}, & C_{g,2}^{scaled} := \gamma^{\frac{1}{2}}C_{g,2}, & D_{g,11}^{scaled} := \gamma^{-1}D_{g,11},\\
  D_{g,12}^{scaled} := D_{g,12}, & D_{g,21}^{scaled} := D_{g,21}, & D_{g,22}^{scaled} := \gamma D_{g,22}.
  \end{array}
  \label{Matrices_Sigma_g_Scaled}
  \end{equation}
Moreover, the transfer function of the system $\Sigma_Q$ needed for the original system can be choosen such that $\mathbf{G}_Q = \gamma\mathbf{G}_Q^{scaled}$, where $\mathbf{G}_Q^{scaled}$ is the transfer function of $\Sigma_Q^{scaled}$ that is choosen in Theorem \ref{Thm_Hinf_Discr_Finite} with $\Vert \mathbf{G}_Q^{scaled}\Vert_{\infty,\overline{\mathbb{D}}^\complement} < 1$.
\item The fact that the matrices $V_c, V_o$ and $V_\times$ can be chosen as lower triangular comes from \cite[Theorem 4.12.9]{IonescuWeiss_99}.
\end{itemize}
\end{remark}


\section{Solution in continuous-time}\label{Solution_Conti}

By taking advantage of the previously described solution to the suboptimal $H^\infty$-control problem for discrete-time finite-dimensional systems, we derive the solution to the suboptimal $H^\infty$-control problem for \eqref{StateEquation_Diag_Uniform}--\eqref{Output_Diag_Uniform} in this section. First note that, by assuming the invertibility of $K$, \eqref{StateEquation_Diag_Uniform}--\eqref{Output_Diag_Uniform} may be written equivalently as 
\begin{align}
  \frac{\partial x}{\partial t}(\zeta,t) &= -\frac{\partial}{\partial \zeta}(\lambda_0(\zeta)x(\zeta,t)),\label{eq:PDEs}\\
  x(\zeta,0) &= x_0(\zeta).\label{StateEquation_Diag_Uniform_Chg}\\
  B_{1,d}d(t) + B_{2,d}u(t) &= \lambda_0(0)x(0,t) - A_d\lambda_0(1)x(1,t),\label{Input_Diag_Uniform_Chg}\\
  z(t) &= C_{1,d}\lambda_0(1)x(1,t) + D_{11,d}d(t) + D_{12,d}u(t)\label{Output_Reg_Diag_Uniform_Chg}\\
  y(t) &= C_{2,d}\lambda_0(1)x(1,t) + D_{21,d}d(t) + D_{22,d}u(t),\label{Output_Diag_Uniform_Chg}
  \end{align}
where the matrices \eqref{Matrices_Discrete} have been used. Before giving the formulation of the suboptimal $H^\infty$-control problem for \eqref{StateEquation_Diag_Uniform_Chg}--\eqref{Output_Diag_Uniform_Chg}, one may ask the following question ``\textit{what could be the form of a dynamic controller $\Sigma_c$ that would be connected to \eqref{eq:PDEs}--\eqref{Output_Diag_Uniform_Chg} as described in Figure \ref{fig:Diagram_Closed_Loop} ?}"

To provide an answer, observe that according to Lemma \ref{lem:discrete-representation}, \eqref{StateEquation_Diag_Uniform_Chg}--\eqref{Output_Diag_Uniform_Chg} may be written equivalently as a discrete-time system where the operators are multiplication operators by constant matrices. Moreover, Theorem \ref{Thm_Hinf_Discr_Finite} gives a parameterization of all the controllers that solves the suboptimal $H^\infty$-control problem described in Definition \ref{Def:Hinf_SubOpt}. According to Section \ref{Hinf_Discrete}, those stabilizing controllers have a representation as finite-dimensional discrete-time systems. For this reason, it is reasonable to think that a dynamic controller given by an infinite-dimensional discrete-time whose operators are multiplication operators by constant matrices would be suitable to solve the suboptimal $H^\infty$-control problem for \eqref{StateEquation_Diag_Uniform_Chg}--\eqref{Output_Diag_Uniform_Chg}. Moreover, as has been pointed out in \cite{CurtainZwart2020}, the class of multiplication operators is able to provide the solution to many different problems for some classes of systems such as the class of spatially invariant systems, see e.g. \cite[Theorem 4.1.4, Lemma 4.3.5, Lemma 9.2.12]{CurtainZwart2020}. Using Lemma \ref{lem:discrete-representation} going from discrete-time to continuous-time, such controllers would have the following form
\begin{align}
  \frac{\partial x_c}{\partial t}(\zeta,t) &= -\frac{\partial}{\partial \zeta}(\lambda_0(\zeta)x_c(\zeta,t)),\,\,\,\,\,\,x_c(\zeta,0) = x_{0,c}(\zeta),\label{State_Controller_PDE}\\
  B_c u_c(t) &= \lambda_0(0)x_c(0,t) - A_c\lambda_0(1)x_c(1,t),\label{Input_Controller_PDE}\\
  y_c(t) &= C_c\lambda_0(1)x_c(1,t) + D_c u_c(t),\label{Output_Controller_PDE}
\end{align}
where $A_c, B_c, C_c$ and $D_c$ are matrices that need to be defined. In \eqref{State_Controller_PDE}--\eqref{Output_Controller_PDE}, the input $u_c(t)$ is the output $y(t)$ of the open-loop system \eqref{StateEquation_Diag_Uniform_Chg}--\eqref{Output_Diag_Uniform_Chg} and the output $y_c(t)$ is the input $u(t)$ of the open-loop system \eqref{StateEquation_Diag_Uniform_Chg}--\eqref{Output_Diag_Uniform_Chg}.

Before giving the definition of suboptimal $H^\infty$-control in this context and before stating the main theorem of this section, let us have a look at the interconnection between \eqref{StateEquation_Diag_Uniform_Chg}--\eqref{Output_Diag_Uniform_Chg} and \eqref{State_Controller_PDE}--\eqref{Output_Controller_PDE}. 
\begin{proposition}\label{Prop:closed-loop_PDE}
Under the assumption that the matrix $S:= I-D_{22,d}D_c$ is invertible, the closed-loop system obtained from \eqref{StateEquation_Diag_Uniform_Chg}--\eqref{Output_Diag_Uniform_Chg} and \eqref{State_Controller_PDE} and the interconnection $u_c(t) = y(t)$ and $y_c(t) = u(t)$, is given by
\begin{align}
\frac{\partial}{\partial t}\left(\begin{matrix}x(\zeta,t)\\ x_c(\zeta,t)\end{matrix}\right) &= -\frac{\partial}{\partial \zeta}\lambda_0(\zeta)\left(\begin{matrix}x(\zeta,t)\\ x_c(\zeta,t)\end{matrix}\right),\,\,\,\,\,\,\,
\left(\begin{matrix}x(\zeta,0)\\x_c(\zeta,0)\end{matrix}\right) = \left(\begin{matrix}x_0(\zeta)\\ x_{0,c}(\zeta)\end{matrix}\right),\label{StateEquation_Closed-loop_PDE}\\
B_{cl}d(t) &= \lambda_0(0)\left(\begin{matrix}x(0,t)\\ x_c(0,t)\end{matrix}\right) - A_{cl}\lambda_0(1)\left(\begin{matrix}x(1,t)\\ x_c(1,t)\end{matrix}\right),\label{Input_Closed-loop_PDE}\\
z(t) &= C_{cl}\lambda_0(1)\left(\begin{matrix}x(1,t)\\ x_c(1,t)\end{matrix}\right) + D_{cl}d(t),\label{Output_Closed-loop_PDE}
\end{align}
where the matrices $A_{cl}, B_{cl}, C_{cl}$ and $D_{cl}$ are given by 
\begin{align}
A_{cl} &= \left(\begin{matrix}A_d + B_{2,d}\tilde{S}^{-1}D_cC_{2,d} & B_{2,d}\tilde{S}^{-1}C_c\\ B_cS^{-1}C_{2,d} & A_c + B_cS^{-1} D_{22,d} C_c\end{matrix}\right)\label{A_cl_PDE}\\
B_{cl} &= \left(\begin{matrix}B_{1,d} + B_{2,d}\tilde{S}^{-1}D_cD_{21,d}\\ B_cS^{-1}D_{21,d}\end{matrix}\right)\label{B_cl_PDE}\\
C_{cl} &= \left(\begin{matrix}C_{1,d}+D_{12,d}D_cS^{-1}C_{2,d} & D_{12,d}\tilde{S}^{-1}C_c\end{matrix}\right)\label{C_cl_PDE}\\
D_{cl} &= \left(D_{11,d} + D_{12,d}D_cS^{-1}D_{21,d}\right),\label{D_cl_PDE}
\end{align}
with $\tilde{S} = I - D_cD_{22,d}$.
\end{proposition}

\begin{proof}[Proof.]
We will only present the main steps of the proof. From \eqref{Output_Diag_Uniform_Chg} and \eqref{Output_Controller_PDE}, we get that
\begin{align*}
\left(I-D_{22,d}D_c\right)y(t) = C_{2,d}\lambda_0(1)x(1,t) + D_{21,d}d(t)  + D_{22,d}C_c\lambda_0(1)x_c(1,t).
\end{align*}
According to the previous equation, \eqref{Input_Diag_Uniform_Chg}--\eqref{Output_Diag_Uniform_Chg} and \eqref{Input_Controller_PDE}--\eqref{Output_Controller_PDE} may be written equivalently as
\begin{align}
&B_{1,d}d(t) + B_{2,d}u(t) = \lambda_0(0)x(0,t) - A_d\lambda_0(1)x(1,t)\label{Inter_1}\\
&z(t) = C_{1,d}\lambda_0(1)x(1,t) + D_{11,d}d(t) + D_{12,d}u(t)\label{Inter_2}\\
&B_cS^{-1}\left(C_{2,d}\lambda_0(1)x(1,t) + D_{21,d}d(t)  + D_{22,d}C_c\lambda_0(1)x_c(1,t)\right) = \lambda_0(0)x_c(0,t) - A_cx_c(1,t)\label{Inter_3}\\
&u(t) = C_c\lambda_0(1)x_c(1,t) + D_cS^{-1}\left(C_{2,d}\lambda_0(1)x(1,t) + D_{21,d}d(t)  + D_{22,d}C_c\lambda_0(1)x_c(1,t)\right).\label{u(t)_inter}
\end{align}
It remains to eliminate the term $u(t)$ in the above equations. This can be done by injecting \eqref{u(t)_inter} in \eqref{Inter_1}--\eqref{Inter_3}. The proof concludes by noting that the two following relations 
\begin{align*}
  \tilde{S}^{-1} = I + D_cS^{-1}D_{22,d},\,\,\, D_cS^{-1} = \tilde{S}^{-1} D_c.
\end{align*}
hold true.
\end{proof}

The aim of this section is to show that we can always find a dynamic controller of the form \eqref{State_Controller_PDE}--\eqref{Output_Controller_PDE} that solves the suboptimal $H^\infty$-control problem for \eqref{StateEquation_Diag_Uniform_Chg}--\eqref{Output_Diag_Uniform_Chg}.

Given \eqref{State_Controller_PDE}--\eqref{Output_Controller_PDE}, we introduce hereafter a precise definition of what is meant by suboptimal $H^\infty$-control for \eqref{StateEquation_Diag_Uniform_Chg}--\eqref{Output_Diag_Uniform_Chg}.

\begin{definition}\label{Def:Hinf_Conti}
Given a positive parameter $\gamma$, the suboptimal $H^\infty$-control problem for \eqref{StateEquation_Diag_Uniform_Chg}--\eqref{Output_Diag_Uniform_Chg} consists in finding a dynamic controller of the form \eqref{State_Controller_PDE}--\eqref{Output_Controller_PDE} such that the closed-loop system has the following properties
\begin{itemize}
  \item it is well-defined, that is, the matrix $S$ (equivalently, $\tilde{S}$) is invertible;
  \item the $C_0$-semigroup generated by the closed-loop homogeneous dynamics is exponentially stable;
  \item the transfer function from the disturance $d$ to the \textit{to-be regulated} output $z$, denoted $\mathbf{G}_{cl}$, satisfies $\mathbf{G}_{cl}\in H^\infty(\mathbb{C}_0^+,\mathbb{C}^{m\times p})$ and has $H^\infty$-norm less than $\gamma$. 
\end{itemize}
\end{definition}
The Hardy space\footnote{The notation $\mathbb{C}_0^+$ stands for the open right-half plane of the complex plane.} $H^\infty(\mathbb{C}_0^+,\mathbb{C}^{m\times p})$ consists in functions $\mathbf{G}:\mathbb{C}_0^+\to\mathbb{C}^{m\times p}$ whose elements are analytic and bounded on $\mathbb{C}_0^+$. The $H^\infty$-norm of a function $\mathbf{G}\in H^\infty(\mathbb{C}_0^+,\mathbb{C}^{m\times p})$ is defined as
  \begin{equation}
  \Vert\mathbf{G}\Vert_{\infty,\mathbb{C}_0^+} := \sup_{\omega\in\mathbb{R}}\Vert \mathbf{G}(i\omega)\Vert. 
  \label{Hinf_Norm_Conti}
  \end{equation}
The shortcut $H^\infty(\mathbb{C}_0^+)$ is used in what follows.

The following theorem is the main theorem of this manuscript and relates the solution to the suboptimal $H^\infty$-control problem stated in Theorem \ref{Thm_Hinf_Discr_Finite} for finite-dimensional discrete-time systems and the solution to the suboptimal $H^\infty$-control problem described in Definition \ref{Def:Hinf_Conti}.

\begin{theorem}\label{Thm:Hinf_conti}
Consider the infinite-dimensional system \eqref{StateEquation_Diag_Uniform_Chg}--\eqref{Output_Diag_Uniform_Chg} with matrices $A_d, B_{1,d}, B_{2,d}, C_{1,d}$, $C_{2,d}, D_{11,d}, D_{12,d}, D_{21,d}$ and $D_{22,d}$ defined in \eqref{Matrices_Discrete}. Assume that the following rank conditions are satisfied
\begin{itemize}
  \item $\rank(\begin{smallmatrix}zI - A_d & B_{2,d}\end{smallmatrix}) = n$ for all $z\in\sigma(A_d)$ such that $\vert z\vert > 1$;
  \item $\rank(\begin{smallmatrix}zI - A_d\\ C_{2,d}\end{smallmatrix}) = n$ for all $z\in\sigma(A_d)$ such that $\vert z\vert > 1$;
  \item $\rank \left(\begin{smallmatrix}e^{j\theta}-A_d & -B_{2,d}\\
    -C_{1,d} & -D_{12,d}\end{smallmatrix}\right) = n + p$ for every $\theta\in [0,2\pi[$;
  \item $\rank\left(\begin{smallmatrix}e^{j\theta}-A_d & -B_{1,d}\\
    -C_{2,d} & -D_{21,d}\end{smallmatrix}\right) = n + k$ for every $\theta\in [0,2\pi[$.
\end{itemize}
Moreover, consider the matrices $A_c, B_c, C_c$ and $D_c$ obtained by applying Theorem \ref{Thm_Hinf_Discr_Finite} to the matrices $A_d, B_d, C_d$ and $D_d$ defined in \eqref{Matrices_Discrete}. Then the closed-loop system \eqref{StateEquation_Closed-loop_PDE}--\eqref{Output_Closed-loop_PDE} satisfies the properties described in Definition \ref{Def:Hinf_Conti} with $\gamma = 1$, that is, the suboptimal $H^\infty$-control problem is solved for \eqref{StateEquation_Diag_Uniform_Chg}--\eqref{Output_Diag_Uniform_Chg} with $\gamma=1$.
\end{theorem}

\begin{proof}[Proof.]
We start by considering the suboptimal $H^\infty$-control problem for the finite-dimensional discrete-time system \eqref{Sigma_Block} where $A_d, B_{1,d}, B_{2,d}, C_{1,d}$, $C_{2,d}, D_{11,d}, D_{12,d}, D_{21,d}$ and $D_{22,d}$ are the matrices defined in \eqref{Matrices_Discrete}. The matrices $A_c, B_c, C_c$ and $D_c$ obtained by applying Theorem \ref{Thm_Hinf_Discr_Finite} to \eqref{Sigma_Block} are such that 
\begin{itemize}
\item the closed-loop system \eqref{Sigma_closed-loop} is well-defined, that is, the matrix $S := I-D_{22,d}D_c$ is invertible;
\item the closed-loop matrix $A_{cl}$ given in \eqref{A_cl} has spectral radius less than $1$;
\item the transfer function of the closed-loop system \eqref{Sigma_closed-loop} from $d$ to $z$, defined by $\mathbf{G}_{cl}^d(z) = D_{cl} + C_{cl}(zI-A_{cl})^{-1}B_{cl}$, has $H^\infty$-norm less than $1$,
\end{itemize}
see Definition \ref{Def:Hinf_SubOpt}. First remark that the invertibility of $S$ is sufficient to write the closed-loop system \eqref{StateEquation_Closed-loop_PDE}--\eqref{Output_Closed-loop_PDE}. Moreover, observe that the closed-loop system \eqref{StateEquation_Closed-loop_PDE}--\eqref{Output_Closed-loop_PDE} is of the same form as \eqref{StateEquation_Diag_Uniform}--\eqref{Output_Diag_Uniform}, which makes it a well-posed boundary controlled system if and only if the operator
\begin{align}
\mathcal{A}f &= -\frac{\d}{\d\zeta}(\lambda_0 f)\label{Op_A_cl}\\
D(\mathcal{A}) &= \left\{f\in \L^2(0,1;\mathbb{R}^r), \lambda_0 f\in\H^1(0,1;\mathbb{R}^r), \lambda_0(0)f(0) - A_{cl}\lambda_0(1)f(1) = 0\right\}\label{D(A_cl)}
\end{align}
is the infinitesimal generator of a $C_0$-semigroup on $\L^2(0,1;\mathbb{R}^r)$, where $r\in\mathbb{N}$ is the dimension of the matrix $A_{cl}$, not known a priori. According to Proposition \ref{prop:well-posed}, this is the case because the matrix in front of $\lambda_0(0)f(0)$ in \eqref{D(A_cl)} is the identity matrix, which is obviously invertible. Let us have a look at the internal stability of the closed-loop system. This question is equivalent to look at whether $\mathcal{A}$ is the generator of an exponentially stable $C_0$-semigroup or not. Thanks to Lemma \ref{lem:discrete-representation}, the closed-loop system \eqref{StateEquation_Closed-loop_PDE}--\eqref{Output_Closed-loop_PDE} admits an equivalent representation as an infinite-dimensional discrete-time system of the form 
\begin{align*}
x_{cl}(j+1)(\zeta) &= A_{cl}x_{cl}(j)(\zeta) + B_{cl}d(j)(\zeta)\\
z_{cl}(j)(\zeta) &= C_{cl}x_{cl}(j)(\zeta) + D_{cl}d(j)(\zeta),
\end{align*}
where, for $j\in\mathbb{N}$, the functions $x_{cl}(j)(\cdot)\in\L^2(0,1;\mathbb{R}^r), d(j)(\cdot)\in\L^2(0,1;\mathbb{R}^k), z_{cl}(j)(\cdot)\in\L^2(0,1;\mathbb{R}^l)$. Because the matrix $A_{cl}$ has spectral radius less than $1$, the semigroup generated by $\mathcal{A}$ is exponentially stable. It remains to look at the transfer function of the closed-loop system \eqref{StateEquation_Closed-loop_PDE}--\eqref{Output_Closed-loop_PDE}. It is given as an operator from $\mathbb{C}^k$ to $\mathbb{C}^l$ by $\mathbf{G}_{cl}(s) = D_{cl} + C_{cl}(e^{sp(1)}I-A_{cl})^{-1}B_{cl}$. By construction, the transfer function of the finite-dimensional discrete-time system \eqref{Sigma_closed-loop}, given by $\mathbf{G}_{cl}^{d}(z) = D_{cl} + C_{cl}\left(zI-A_{cl}\right)^{-1}B_{cl}$ is in $H^\infty(\overline{\mathbb{D}}^\complement,\mathbb{C}^{k\times l})$ and $A_{cl}$ has spectral radius less than $1$. As a consequence, $\mathbf{G}_{cl}(s)$ is analytic and bounded whenever $\vert e^{sp(1)}\vert > 1$. Because $p(1) >0$, $\vert e^{sp(1)}\vert > 1$ is equivalent to $\mathrm{Re}(s)>0$. This implies that $\mathbf{G}_{cl}\in H^\infty(\mathbb{C}_0^+,\mathbb{C}^{m\times p})$. According to Theorem \ref{Thm_Hinf_Discr_Finite}, the $H^\infty$-norm of $\mathbf{G}_{cl}^d$ is less than one, that is 
\begin{align*}
\Vert\mathbf{G}_{cl}^d\Vert_{\infty,\overline{\mathbb{D}}^\complement} = \sup_{\theta\in [0,2\pi]}\Vert \mathbf{G}_{cl}^d(e^{i\theta})\Vert < 1.
\end{align*}
Observe that
\begin{align*}
\Vert\mathbf{G}_{cl}\Vert_{\infty,\mathbb{C}_0^+} = \sup_{\omega\in\mathbb{R}} \Vert\mathbf{G}_{cl}(i\omega)\Vert &=\sup_{\omega\in\mathbb{R}} \Vert D_{cl} + C_{cl}\left(e^{i\omega p(1)}I-A_{cl}\right)^{-1}B_{cl}\Vert\\
&= \sup_{\theta\in [0,2\pi]} \Vert D_{cl} + C_{cl}\left(e^{i\theta}I-A_{cl}\right)^{-1}B_{cl}\Vert\\
&= \sup_{\theta\in [0,2\pi]}\Vert \mathbf{G}_{cl}^d(e^{i\theta})\Vert\\
&= \Vert \mathbf{G}_{cl}^d\Vert_{\infty,\overline{\mathbb{D}}^\complement}\\
&< 1,
\end{align*}
which concludes the proof.
\end{proof}

\begin{remark}
The two first rank conditions assumed in Theorem \ref{Thm:Hinf_conti} implies that Assumption \ref{Assum:Stab_Detec} is satisfied whereas the two other rank conditions imply that Assumption \ref{Assum:Regularity} holds.
\end{remark}


\section{Example: a vibrating string with force control and velocity measurement}\label{Example}

In this section, we apply the theory developed earlier in this manuscript to a one-dimensional vibrating string whose force is controlled at the boundary and whose velocity is measured at the boundary as well. This model is governed by the following PDE
\begin{equation}
\frac{\partial^2 w}{\partial t^2}(\zeta,t) = \frac{1}{\rho(\zeta)}\frac{\partial}{\partial\zeta}\left(T(\zeta)\frac{\partial w}{\partial\zeta}(\zeta,t)\right),
\label{eq:PDE_string}
\end{equation}
where $w(\zeta,t)$ is the vertical displacement of the string at position $\zeta\in [0,1]$ and at time $t\geq 0$. $\rho$ and $T$ are the mass density and the Young's modulus of the string, respectively, and are positive functions. We assume them constant in what follows, i.e. $\rho(\zeta) = 0$ and $T(\zeta) = T$. We assume that the velocity at $\zeta=1$ is perturbed and that a force is applied to the string at $\zeta = 0$, which results in the following disturbance and control
\begin{align}
d(t) &= \frac{\partial w}{\partial t}(1,t)\label{eq:Disturb_string}\\
u(t) &= T\frac{\partial w}{\partial\zeta}(0,t)\label{eq:Control_string},
\end{align}
respectively. Moreover, the \textit{to-be regulated} output is supposed to be the force of the string at $\zeta=1$ and the velocity of the string is measured at $\zeta=0$, which is written as
\begin{align}
z(t) &= T\frac{\partial w}{\partial\zeta}(1,t)\label{eq:RegOutput_string}\\
y(t) &= \frac{\partial w}{\partial t}(0,t)\label{eq:Output_string}.
\end{align}
Using the energy variables $x_1(\zeta,t) = \rho\frac{\partial w}{\partial t}(\zeta,t), x_2(\zeta,t) = \frac{\partial w}{\partial\zeta}(\zeta,t)$, the PDE \eqref{eq:PDE_string} together with the boundary inputs and outputs \eqref{eq:Disturb_string}--\eqref{eq:Control_string} and \eqref{eq:RegOutput_string}--\eqref{eq:Output_string} may be written equivalently as 
\begin{align}
\frac{\partial x}{\partial t}(\zeta,t) &= P_1\mathcal{H}\frac{\partial x}{\partial\zeta}(\zeta,t)\label{eq:PDE_String_EnergyVar}\\
\left(\begin{matrix}1\\ 0\end{matrix}\right)d(t) + \left(\begin{matrix}0\\ 1\end{matrix}\right)u(t) &= \left(\begin{matrix}0 & 0\\0 & 1\end{matrix}\right)\mathcal{H}x(0,t) + \left(\begin{matrix}1 & 0\\0 & 0\end{matrix}\right)\mathcal{H}x(1,t)\label{eq:InputDistr_string_energyVar}\\
z(t) &= \left(\begin{matrix}0 & 1\end{matrix}\right)\mathcal{H}x(1,t)\label{eq:RegOut_string_energyVar}\\
y(t) &= \left(\begin{matrix}1 & 0\end{matrix}\right)\mathcal{H}x(0,t)\label{eq:Out_string_energyVar},
\end{align}
where $P_1 := \left(\begin{smallmatrix}0 & 1\\1 & 0\end{smallmatrix}\right)$ and $\mathcal{H} := \left(\begin{smallmatrix}\frac{1}{\rho} & 0\\ 0 & T\end{smallmatrix}\right)$. Observe that the matrix $P_1\mathcal{H}$ may be decomposed as $P_1\mathcal{H} =  S^{-1}\Delta S$, where 
\begin{equation*}
S = \left(\begin{smallmatrix}\frac{1}{2} & \frac{1}{2\sigma}\\ -\frac{1}{2} & \frac{1}{2\sigma}\end{smallmatrix}\right), \Delta = \left(\begin{smallmatrix}v & 0\\ 0 & -v\end{smallmatrix}\right), S^{-1} = \left(\begin{smallmatrix}1 & -1\\\sigma & \sigma\end{smallmatrix}\right),
\end{equation*}
with $v := \sqrt{\frac{T}{\rho}}$ and $\sigma := \frac{1}{\sqrt{\rho T}}$. We define the variable $\tilde{x}(\zeta,t) := Sx(\zeta,t)$. Hence, \eqref{eq:PDE_String_EnergyVar}--\eqref{eq:Out_string_energyVar} become
\begin{align}
\frac{\partial \tilde{x}}{\partial t}(\zeta,t) &= \Delta \frac{\partial \tilde{x}}{\partial\zeta}(\zeta,t)\label{eq:PDE_String_EnergyVar_bis}\\
\left(\begin{matrix}1\\ 0\end{matrix}\right)d(t) + \left(\begin{matrix}0\\ 1\end{matrix}\right)u(t) &= v\left(\begin{matrix}0 & 0\\1 & 1\end{matrix}\right)\tilde{x}(0,t) + v\left(\begin{matrix}\sigma & -\sigma\\0 & 0\end{matrix}\right)\tilde{x}(1,t)\nonumber\\
z(t) &= v\left(\begin{matrix}1 & 1\end{matrix}\right)\tilde{x}(1,t)\nonumber\\
y(t) &= v\left(\begin{matrix}\sigma & -\sigma\end{matrix}\right)\tilde{x}(0,t)\nonumber,
\end{align}
where the relation $\mathcal{H}S^{-1} = v\left(\begin{smallmatrix}\sigma & -\sigma\\1 & 1\end{smallmatrix}\right)$ has been used. In order to get the same velocity for each transport PDE in \eqref{eq:PDE_String_EnergyVar_bis}, we introduce the variable $\overline{x}_1(\zeta,t) := \tilde{x}_1(1-\zeta,t)$. By defining the state vector $X(\zeta,t) := \left(\begin{smallmatrix}\overline{x}_1(\zeta,t) & \tilde{x}_2(\zeta,t)\end{smallmatrix}\right)^T$, there holds
\begin{align}
\frac{\partial X}{\partial t}(\zeta,t) &= -v\frac{\partial X}{\partial\zeta}(\zeta,t)\label{eq:PDE_String_EnergyVar_ter}\\
\left(\begin{matrix}1\\ 0\end{matrix}\right)d(t) + \left(\begin{matrix}0\\ 1\end{matrix}\right)u(t) &= -v\left(\begin{matrix}-\sigma & 0\\0 & -1\end{matrix}\right)X(0,t) - v\left(\begin{matrix}0 & \sigma\\-1 & 0\end{matrix}\right)X(1,t)\label{eq:InputDistr_string_energyVar_ter}\\
z(t) &= -v\left(\begin{matrix}-1 & 0\end{matrix}\right)X(0,t) - v\left(\begin{matrix}0 & -1\end{matrix}\right)X(1,t)\label{eq:RegOut_string_energyVar_ter}\\
y(t) &= -v\left(\begin{matrix}0 & \sigma\end{matrix}\right)X(0,t) - v\left(\begin{matrix}-\sigma & 0\end{matrix}\right)X(1,t)\label{eq:Out_string_energyVar_ter}.
\end{align}
According to the notations introduced in \eqref{Input_Diag_Uniform}--\eqref{Output_Diag_Uniform}, we have that
\begin{equation*}
\begin{array}{ll}
E = \left(\begin{smallmatrix}1\\0\end{smallmatrix}\right), & \\
K = \left(\begin{smallmatrix}-\sigma & 0\\0 & -1\end{smallmatrix}\right), & L = \left(\begin{smallmatrix}0 & \sigma\\-1 & 0\end{smallmatrix}\right),\\
K_y = \left(\begin{smallmatrix}0 & \sigma\end{smallmatrix}\right), & L_y = \left(\begin{smallmatrix}-\sigma & 0\end{smallmatrix}\right),\\
K_z = \left(\begin{smallmatrix}-1 & 0\end{smallmatrix}\right), & L_z = \left(\begin{smallmatrix}0 & -1\end{smallmatrix}\right).
\end{array}
\end{equation*}
Thanks to the positivity of $\rho$ and $T$, the constant $\sigma$ is positive, which implies the well-posedness of the boundary controlled system \eqref{eq:PDE_String_EnergyVar_ter}--\eqref{eq:Out_string_energyVar_ter}, see Proposition \ref{prop:well-posed}. According to \eqref{Matrices_Discrete}, there holds
\begin{equation}
\begin{array}{lll}
A_d = \left(\begin{smallmatrix}0 & 1\\-1 & 0\end{smallmatrix}\right), & B_{1,d} = \left(\begin{smallmatrix}\frac{1}{\sigma}\\ 0\end{smallmatrix}\right), & B_{2,d} = \left(\begin{smallmatrix}0\\ 1\end{smallmatrix}\right),\\
C_{1,d} = \left(\begin{smallmatrix}0 & 2\end{smallmatrix}\right), & C_{2,d} = \left(\begin{smallmatrix}2\sigma & 0\end{smallmatrix}\right), & D_{11,d} = \frac{1}{\sigma},\\
D_{12,d} = 0, & D_{21,d} = 0, & D_{22,d} = -\sigma.  
\end{array}
\label{eq:Matrices_Discrete-String}
\end{equation}
In order to solve the suboptimal control problem for \eqref{eq:PDE_String_EnergyVar_ter}--\eqref{eq:Out_string_energyVar_ter} with an arbitrary $H^\infty$ performance $\gamma > 0$, we will apply Theorem \ref{Thm_Hinf_Discr_Finite} to the following scaled matrices\footnote{The upper script ``s" is used to denote the term ``scaled".}, see \eqref{Matrices_Scaled},
\begin{equation}
  \begin{array}{lll}
  A_d^s = \left(\begin{smallmatrix}0 & 1\\-1 & 0\end{smallmatrix}\right), & B_{1,d}^s = \left(\begin{smallmatrix}\frac{1}{\sigma\sqrt{\gamma}}\\ 0\end{smallmatrix}\right), & B_{2,d}^s = \left(\begin{smallmatrix}0\\ \sqrt{\gamma}\end{smallmatrix}\right),\\
  C_{1,d}^s = \left(\begin{smallmatrix}0 & \frac{2}{\sqrt{\gamma}}\end{smallmatrix}\right), & C_{2,d}^s = \left(\begin{smallmatrix}2\sigma\sqrt{\gamma} & 0\end{smallmatrix}\right), & D_{11,d}^s = \frac{1}{\sigma\gamma},\\
  D_{12,d}^s = 0, & D_{21,d}^s = 0, & D_{22,d}^s = -\sigma\gamma.  
  \end{array}
  \label{eq:Matrices_Discrete-String_scaled}
  \end{equation}
We start by checking Assumption \ref{Assum:Stab_Detec}. For this, observe that
\begin{align*}
\rank\left(\begin{smallmatrix}\lambda I-A_d^s & B_{2,d}^s\end{smallmatrix}\right) = \rank\left(\begin{smallmatrix}\lambda & -1 & 0\\1 & \lambda & \sqrt{\gamma}\end{smallmatrix}\right) = 2,
\end{align*}
for every $\lambda\in\mathbb{C}$, which makes $(A_d,B_{2,d})$ stabilizable. In addition, there holds
\begin{align*}
\rank\left(\begin{smallmatrix}\lambda I-A_d^s\\C_{2,d}^s\end{smallmatrix}\right) = \rank\left(\begin{smallmatrix}\lambda & -1\\ 1 & \lambda\\2\sigma\sqrt{\gamma} & 0\end{smallmatrix}\right) = 2,
\end{align*}
for every $\lambda\in\mathbb{C}$, which means that $(C_{2,d}, A_d)$ is detectable. Moreover, we have that
\begin{align*}
&\rank\left(\begin{smallmatrix}e^{j\theta}-A_d^s & -B_{2,d}^s\\-C_{1,d}^s & -D_{12,d}^s\end{smallmatrix}\right) = \rank\left(\begin{smallmatrix}e^{i\theta} & -1 & 0\\1 & e^{i\theta} & -\sqrt{\gamma}\\0 & \frac{-2}{\sqrt{\gamma}} & 0\end{smallmatrix}\right) = 3\\
&\rank\left(\begin{smallmatrix}e^{i\theta}-A_d^s & -B_{1,d}^s\\-C_{2,d}^s & -D_{21,d}^s\end{smallmatrix}\right) = \rank\left(\begin{smallmatrix}e^{i\theta} & -1 & -\frac{1}{\sigma\sqrt{\gamma}}\\1 & e^{i\theta} & 0\\-2\sigma\sqrt{\gamma} & 0 & 0\end{smallmatrix}\right) = 3
\end{align*}
for every $\theta\in[0,2\pi]$. This implies that Assumptions \ref{Assum:Stab_Detec} and \ref{Assum:Regularity} are satisfied. Now we apply Theorem \ref{Thm_Hinf_Discr_Finite} with the matrices given in \eqref{eq:Matrices_Discrete-String}. We start by computing the matrices $Q_c^s, L_c^s, R_c^s, Q_o^s, L_o^s$ and $R_o^s$ defined in the statement of Theorem \ref{Thm_Hinf_Discr_Finite}. We have
\begin{equation*}
\begin{array}{lll}
Q_c^s = \left(\begin{smallmatrix}0 & 0\\ 0 & \frac{4}{\gamma}\end{smallmatrix}\right), & L_c^s = \left(\begin{smallmatrix}0 & 0\\\frac{2}{\sigma\gamma\sqrt{\gamma}} & 0\end{smallmatrix}\right), & R_c^s = \left(\begin{smallmatrix}\frac{1}{\sigma^2\gamma^2}-1 & 0\\0 & 0\end{smallmatrix}\right),\\
Q_o^s = \left(\begin{smallmatrix}\frac{1}{\sigma^2\gamma} & 0\\0 & 0\end{smallmatrix}\right), &L_o^s = \left(\begin{smallmatrix}\frac{1}{\sigma^2\gamma\sqrt{\gamma}} & 0\\0 & 0\end{smallmatrix}\right), & R_o^s = \left(\begin{smallmatrix}\frac{1}{\sigma^2\gamma^2}-1 & 0\\0 & 0\end{smallmatrix}\right).
\end{array}
\end{equation*}
It can be shown that the matrices 
\begin{align}
X^s &:= \left(\begin{smallmatrix}0 & 0\\0 & \frac{4\sigma^2\gamma}{\sigma^2\gamma^2-1}\end{smallmatrix}\right)\label{X_string_Riccati}\\
Y^s &:= \left(\begin{smallmatrix}\frac{\gamma}{\sigma^2\gamma^2-1} & 0\\0 & 0\end{smallmatrix}\right)\label{Y_string_Riccati}
\end{align}
are solutions to the Riccati equation associated with the Kalman-Szego-Popov-Yakubovich systems $\Sigma_*^s$ and $\Sigma_o^s$, respectively, see Definition \ref{Def:KSPYS}. It is then easy to see that $\gamma$ has to be larger than $\frac{1}{\sigma}$ for those matrices to be semi-positive definite. We make this assumption in what follows. 
Let us see whether these obtained solutions are stabilizing or not. With the notation $F_c^s := -(R_c^s + (B_d^s)^TX^sB_d^s)^{-1}((B_d^s)^TX^sA_d^s + (L_c^s)^T)$ and $F_o := -(R_o^s+C_d^sY^s(C_d^s)^T)^{-1}(C_d^sY^s(A_d^s)^T + (L_o^s)^T)$, observe that
\begin{equation*} 
\begin{array}{ll}
A_d^s + B_d^sF_c^s = \left(\begin{smallmatrix}0 & \frac{\sigma^2+1}{\sigma^2-1}\\0 & 0\end{smallmatrix}\right), & (A_d^s)^T + (C_d^s)^TF_o^s = \left(\begin{smallmatrix}0 & 0\\\frac{\sigma^2+1}{\sigma^2-1} & 0\end{smallmatrix}\right),
\end{array}
\end{equation*}
which are both stable matrices. Hence, the solutions $X^s$ and $Y^s$ are stabilizing. From \eqref{X_string_Riccati} and \eqref{Y_string_Riccati} we can compute the matrices $V_c^s, W_c^s$ and $V_o^s, W_o^s$ solution of \eqref{EqV_KSPYS}--\eqref{EqW_KSPYS}. There holds
\begin{equation*}
\begin{array}{ll}
V_c^s = \left(\begin{smallmatrix}\frac{\sqrt{\sigma^2\gamma^2-1}}{\sigma\gamma} & 0\\0 & \frac{2\sigma\gamma}{\sqrt{\sigma^2\gamma^2-1}}\end{smallmatrix}\right), & V_o^s = \left(\begin{smallmatrix}\frac{\sqrt{\sigma^2\gamma^2-1}}{\sigma\gamma} & 0\\0 & \frac{2\sigma\gamma}{\sqrt{\sigma^2\gamma^2-1}}\end{smallmatrix}\right),\\
W_c^s = \left(\begin{smallmatrix}0 & \frac{-2}{\sqrt{\gamma}\sqrt{\sigma^2\gamma^2-1}}\\\frac{-2\sigma\sqrt{\gamma}}{\sqrt{\sigma^2\gamma^2-1}} & 0\end{smallmatrix}\right), & W_o^s = \left(\begin{smallmatrix}\frac{-1}{\sigma\sqrt{\gamma}\sqrt{\sigma^2\gamma^2-1}} & 0\\0 & \frac{-\sqrt{\gamma}}{\sqrt{\sigma^2\gamma^2-1}}\end{smallmatrix}\right).
\end{array}
\end{equation*}
Observe now that the product $X^sY^s$ is given by the null matrix, whose spectral radius is obviously less than $1$. As a consequence, the suboptimal $H^\infty$-control problem described in Definition \ref{Def:Hinf_SubOpt} has a solution for $\gamma > \frac{1}{\sigma}$, see Theorem \ref{Thm_Hinf_Discr_Finite}. Now we look at $\Sigma_\times^s$. There holds
\begin{equation*}
\begin{array}{lll}
Q_\times^s = \left(\begin{smallmatrix}\frac{\gamma}{\sigma^2\gamma^2-1} & 0\\0 & 0\end{smallmatrix}\right), & L_\times^s = \left(\begin{smallmatrix}0 & 0\\0 & 0\end{smallmatrix}\right), & R_\times^s = \left(\begin{smallmatrix}-1 & 0\\0 & 0\end{smallmatrix}\right).
\end{array}
\end{equation*}
According to Theorem \ref{Thm_Hinf_Discr_Finite}, the solution $Z^s$ to the Riccati equation associated with the Kalman-Szego-Popov-Yakubovich system $\Sigma_\times$ and $J_\times$ is given by $Z^s := Y^s(I-X^sY^s)^{-1} = Y^s$. Moreover, $V_\times^s$ and $W_\times^s$ are expressed as
\begin{equation*}
\begin{array}{ll}
V_\times^s = \left(\begin{smallmatrix}1 & 0\\\frac{-2\sigma\gamma}{\sigma^2\gamma^2-1} & \frac{2\sigma\gamma}{\sqrt{\sigma^2\gamma^2-1}}\end{smallmatrix}\right), &W_\times^s = \left(\begin{smallmatrix}0 & 0\\0 & \frac{-\sqrt{\gamma}}{\sqrt{\sigma^2\gamma^2-1}}\end{smallmatrix}\right).
\end{array} 
\end{equation*}
According to the notations in Theorem \ref{Thm_Hinf_Discr_Finite}, $C_{2,F_1}, S_c$ and $S_\times$ are given by 
\begin{equation*}
\begin{array}{lll}
C_{2,F_1} = \left(\begin{smallmatrix}2\sigma\sqrt{\gamma} & 0\end{smallmatrix}\right), &S_c = \frac{\sigma^2\gamma^2}{\sigma^2\gamma^2-1}, & S_\times = 1.
\end{array}
\end{equation*}
Now we give the matrices that constitute $\Sigma_g^s$, see Theorem \ref{Thm_Hinf_Discr_Finite}. There holds
\begin{equation*} 
\begin{array}{lll}
A_g^s = \left(\begin{smallmatrix}0 & \frac{\sigma^2\gamma^2+1}{\sigma^2\gamma^2-1}\\0 & 0\end{smallmatrix}\right), & B_{g,1}^s = \left(\begin{smallmatrix}0\\0\end{smallmatrix}\right), & B_{g,2}^s = \left(\begin{smallmatrix}0\\ \frac{-\sqrt{\sigma^2\gamma^2-1}}{2\sigma\sqrt{\gamma}}\end{smallmatrix}\right),\\
C_{g,1}^s = \left(\begin{smallmatrix}0 & 0\end{smallmatrix}\right), & C_{g,2}^s = \left(\begin{smallmatrix}\frac{\sqrt{\sigma^2\gamma^2-1}}{\sqrt{\gamma}} & 0\end{smallmatrix}\right), &D_{g,11}^s = \frac{1}{2\sigma\gamma},\\
D_{g,12}^s = \frac{\sqrt{\sigma^2\gamma^2-1}}{2\sigma\gamma}, &D_{g,21}^s = \frac{\sqrt{\sigma^2\gamma^2-1}}{2\sigma\gamma}, &D_{g,22}^s = 0.
\end{array}
\end{equation*}
According to Remark \ref{rem:ScaledSystem_Sigma_g}, the matrices of the system $\Sigma_g$ for the original system are given by 
\begin{equation*} 
  \begin{array}{lll}
  A_g = \left(\begin{smallmatrix}0 & \frac{\sigma^2\gamma^2+1}{\sigma^2\gamma^2-1}\\0 & 0\end{smallmatrix}\right), & B_{g,1} = \left(\begin{smallmatrix}0\\0\end{smallmatrix}\right), & B_{g,2} = \left(\begin{smallmatrix}0\\ \frac{-\sqrt{\sigma^2\gamma^2-1}}{2\sigma\gamma}\end{smallmatrix}\right),\\
  C_{g,1} = \left(\begin{smallmatrix}0 & 0\end{smallmatrix}\right), & C_{g,2} = \left(\begin{smallmatrix}\frac{\sqrt{\sigma^2\gamma^2-1}}{\gamma} & 0\end{smallmatrix}\right), &D_{g,11} = \frac{1}{2\sigma},\\
  D_{g,12} = \frac{\sqrt{\sigma^2\gamma^2-1}}{2\sigma\gamma}, &D_{g,21} = \frac{\sqrt{\sigma^2\gamma^2-1}}{2\sigma\gamma}, &D_{g,22} = 0.
  \end{array}
  \end{equation*}
Now we need to choose a system $\Sigma_Q = (A_Q, B_Q, C_Q, D_Q)$ with rational transfer function $\mathbf{G}_Q\in H^\infty(\overline{\mathbb{D}}^\complement,\mathbb{C})$, which is analytic on $\mathbb{C}\setminus\overline{\mathbb{D}}$ and that satisfies $\Vert\mathbf{G}_Q\Vert_{\infty,\overline{\mathbb{D}}^\complement} < \gamma$. For this, we choose $A_Q<0$ and $B_Q, C_Q\in\mathbb{R}$ such that 
\begin{align*}
\left\vert\frac{B_QC_Q}{1+A_Q}\right\vert<\gamma,
\end{align*}
and $D_Q = 0$. According to Theorem \ref{Thm_Hinf_Discr_Finite}, an optimal compensator $\Sigma_c$ is given by $\Sigma_c = \mathrm{LFT}(\Sigma_g,\Sigma_Q)$, whose matrices $A_c, B_c, C_c$ and $D_c$ are given by
\begin{equation*} 
\begin{array}{ll}
A_c = \left(\begin{smallmatrix}0 & \frac{\sigma^2\gamma^2+1}{\sigma^2\gamma^2-1} & 0\\0 & 0 & -C_Q\frac{\sqrt{\sigma^2\gamma^2-1}}{2\sigma\gamma}\\B_Q\frac{\sqrt{\sigma^2\gamma^2-1}}{\gamma} & 0 & A_Q\end{smallmatrix}\right),& B_c = \left(\begin{smallmatrix}0\\0\\B_Q\frac{\sqrt{\sigma^2\gamma^2-1}}{2\sigma\gamma}\end{smallmatrix}\right),\\
C_c = \left(\begin{smallmatrix}0 & 0 & C_Q\frac{\sqrt{\sigma^2\gamma^2-1}}{2\sigma\gamma}\end{smallmatrix}\right), & D_c = \frac{1}{2\sigma}.
\end{array} 
\end{equation*}
As is mentioned in Remark \ref{rem:D_22_nonzero}, the aforementioned controller does not solve the suboptimal $H^\infty$-control problem in its present form because the matrix $D_{22,d}$ is not the null matrix. Therefore, according to Remark \ref{rem:D_22_nonzero}, the transfer function of the optimal compensator is given by 
\begin{align*}
\mathbf{G}_c^{opt} = \mathbf{G}_c\left(I+D_{22,d}\mathbf{G}_c\right)^{-1} = \mathbf{G}_c\left(I-\sigma\mathbf{G}_c\right)^{-1},  
\end{align*}
where $\mathbf{G}_c$ is the transfer function of $\Sigma_c$. According to \cite[Section 2.1.5]{IonescuWeiss_99}, a possible realization for $\mathbf{G}_c^{opt}$ is given by
\begin{equation*}
\begin{array}{ll}
A_c^{opt} = \left(\begin{smallmatrix}0 & \frac{\sigma^2\gamma^2+1}{\sigma^2\gamma^2-1} & 0 & 0\\0 & 0 & -C_Q\frac{\sqrt{\sigma^2\gamma^2-1}}{2\sigma\gamma} & 0\\B_Q\frac{\sqrt{\sigma^2\gamma^2-1}}{\gamma} & 0 & A_Q + B_QC_Q\frac{\sigma^2\gamma^2-1}{2\sigma\gamma^2} & 0\\0 & 0 & \tilde{B}C_Q\frac{\sqrt{\sigma^2\gamma^2-1}}{\sigma\gamma} & \tilde{A}\end{smallmatrix}\right), & B_c^{opt} = \left(\begin{smallmatrix}0\\0\\B_Q\frac{\sqrt{\sigma^2\gamma^2-1}}{\sigma\gamma}\\\tilde{B}\frac{1}{\sigma}\end{smallmatrix}\right),\\
C_c^{opt} = \left(\begin{smallmatrix}0 & 0 & C_Q\frac{\sqrt{\sigma^2\gamma^2-1}}{\sigma\gamma} & 0\end{smallmatrix}\right), & D_c^{opt} = \frac{1}{\sigma}.
\end{array} 
\end{equation*}
Note that this realization has been computed by considering $\sigma$ as the transfer function of a discrete-time system $(\tilde{A},\tilde{B},\tilde{C},\tilde{D})$ with $\tilde{A}, \tilde{B}$ real numbers such that $\vert\tilde{A}\vert<1, \tilde{C} = 0$ and $\tilde{D} = \sigma$. For more details, we refer to \cite[Section 2.1.5, \textit{Feedback connection}]{IonescuWeiss_99}. A closer look at the matrix $I - D_{22,d}D_c^{opt}$ shows that
\begin{align*}
I - D_{22,d}D_c^{opt} = 1 + \sigma\frac{1}{\sigma} = 2,
\end{align*}
which means that the closed-loop system composed of \eqref{eq:PDE_String_EnergyVar_ter}--\eqref{eq:Out_string_energyVar_ter} and \eqref{State_Controller_PDE}--\eqref{Output_Controller_PDE} with $A_c, B_c, C_c$ and $D_c$ replaced by $A_c^{opt}, B_c^{opt}, C_c^{opt}$ and $D_c^{opt}$ is well-posed. The solution to the suboptimal $H^\infty$-control problem for the vibrating string \eqref{eq:PDE_String_EnergyVar_ter}--\eqref{eq:Out_string_energyVar_ter} is summarized in the following proposition.

\begin{proposition}\label{prop:summary_string}
Let $\sigma > 0$, $A_Q < 0, B_Q, C_Q, \tilde{A}$ and $\tilde{B}$ be real numbers with $\vert A_Q\vert < 1$ and $\vert\tilde{A}\vert < 1$. Then, the suboptimal $H^\infty$-control problem for \eqref{eq:PDE_String_EnergyVar_ter}--\eqref{eq:Out_string_energyVar_ter} admits a solution with prescribed $H^\infty$-performance $\gamma$ provided that $\gamma > \frac{1}{\sigma}$ and $\left\vert\frac{B_QC_Q}{1+A_Q}\right\vert<\gamma$. Moreover, a possible optimal compensator is given by \eqref{State_Controller_PDE}--\eqref{Output_Controller_PDE} in which the matrices $A_c, B_c, C_c$ and $D_c$ are replaced by $A_c^{opt}, B_c^{opt}, C_c^{opt}, D_c^{opt}$, respectively. In addition, the closed-loop system is given by \eqref{StateEquation_Closed-loop_PDE}--\eqref{Output_Closed-loop_PDE} where the matrices $A_{cl}, B_{cl}, C_{cl}$ and $D_{cl}$ are expressed as
\begin{equation*}
\begin{array}{ll}
A_{cl} = \left(\begin{smallmatrix}0 & 1 & 0 & 0 & 0 & 0\\0 & 0 & 0 & 0 & C_Q\frac{\sqrt{\sigma^2\gamma^2-1}}{2\sigma\gamma} & 0\\0 & 0 & 0 & \frac{\sigma^2\gamma^2+1}{\sigma^2\gamma^2-1} & 0 & 0\\0 & 0 & 0 & 0 & -C_Q\frac{\sqrt{\sigma^2\gamma^2-1}}{2\sigma\gamma} & 0\\ B_Q\frac{\sqrt{\sigma^2\gamma^2-1}}{\gamma} & 0 & B_Q\frac{\sqrt{\sigma^2\gamma^2-1}}{\gamma} & 0 & A_Q & 0\\\tilde{B} & 0 & 0 & 0 & \tilde{B}C_Q\frac{\sigma^2\gamma^2-1}{2\sigma\gamma} & \tilde{A}
\end{smallmatrix}\right), & B_{cl} = \left(\begin{smallmatrix} \frac{1}{\sigma}\\ 0\\ 0\\ 0\\ 0\\ 0\\ 0\end{smallmatrix}\right),\\
C_{cl} = \left(\begin{smallmatrix}0 & 2 & 0 & 0 & 0 & 0\end{smallmatrix}\right), & D_{cl} = \frac{1}{\sigma}.
\end{array}
\end{equation*}
Furthermore, the closed-loop transfer function is given by 
\begin{align}
\mathbf{G}_{cl}^{opt}(z) &= D_{cl} + C_{cl}(zI-A_{cl})^{-1}B_{cl} = \frac{1}{\sigma} + \frac{1}{\sigma}\left[\frac{(\sigma^2\gamma^2-1)B_QC_Q}{\sigma\gamma^2z^2(z-A_Q) + B_QC_Q}\right].
\label{eq:Closed-loop_TrF}
\end{align}
\end{proposition}
Note that the matrices depicted in Proposition \ref{prop:summary_string} have been computed by using \eqref{A_cl_PDE}, \eqref{B_cl_PDE}, \eqref{C_cl_PDE} and \eqref{D_cl_PDE}. As an example, we consider the following specific values for $\gamma, A_Q, B_Q, C_Q, \tilde{A}$ and $\tilde{B}$
\begin{equation*} 
\begin{array}{llllll}
\sigma = 6,& A_Q = -\frac{1}{2}, & B_Q = \frac{1}{2}\sqrt{5}, & C_Q = \frac{9}{10}\sqrt{5}, & \tilde{A} = \frac{1}{4}, & \tilde{B} = 2.
\end{array}
\end{equation*}
We shall check that the corresponding closed-loop system satisfies the required properties in the solution of the suboptimal $H^\infty$-control problem for $\gamma=\frac{1}{5}$. First observe that $\sigma>\frac{1}{\gamma}$. With this choice, the eigenvalues of $A_{cl}$ are given by $0.25, -0.9319, 0.2159 + 0.5965i, 0.2159 - 0.5965i, 0, 0$, which are all in the unit circle of the complex plane. This means that the closed-loop system is internally stable. Now we examine the transfer function given in \eqref{eq:Closed-loop_TrF} for which we let the parameter $\sigma$ to be free. The representation of $\Vert\mathbf{G}_{cl}^{opt}(e^{i\theta})\Vert$ as a function of $\theta$ and $\sigma$ is given in Figure \ref{fig:trf_fct}, where the values on the $\sigma$-axis have been selected as being larger than $\frac{1}{\gamma}$. Therein, it can be observed that the supremum of the function is below $\gamma = \frac{1}{5}$, which confirms what we found theoretically.

\begin{figure}
\centering
\includegraphics[scale=0.6]{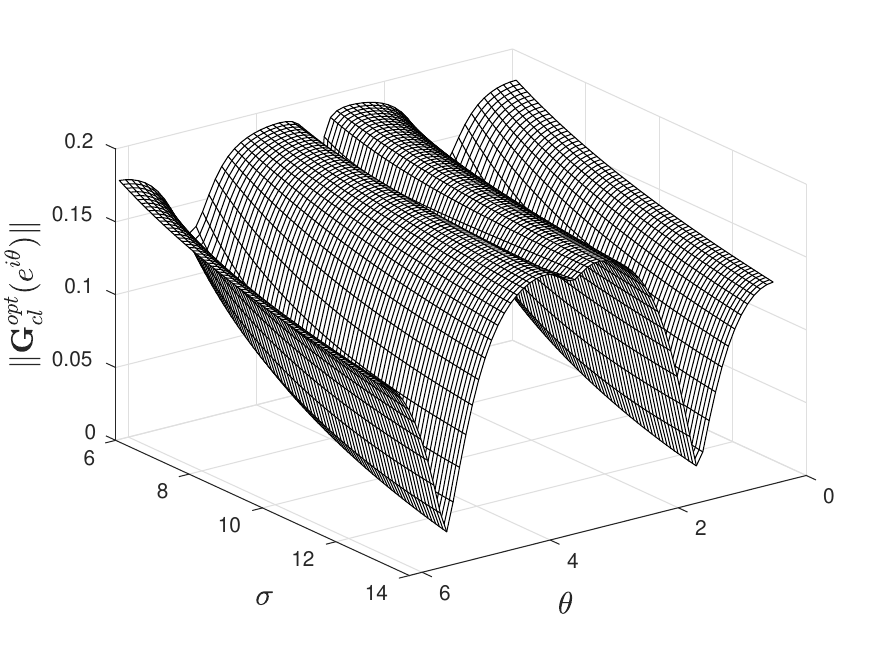}
\caption{Representation of $\Vert\mathbf{G}_{cl}^{opt}(e^{i\theta})\Vert$ for $\theta\in [0,2\pi]$ and $\sigma\in[6.1,13.1]$. \label{fig:trf_fct}}
\end{figure}


\section{Conclusion}\label{Ccl}
The suboptimal $H^\infty$-control problem has been studied for a class of boundary controlled hyperbolic PDEs. A particular feature of that class of systems is that it admits an equivalent representation as an infinite-dimensional discrete-time system. This together with the well-posedness analysis is studied in Section \ref{SystProp}. In Section \ref{Hinf_Discrete}, a way of solving the suboptimal $H^\infty$-control problem for finite-dimensional discrete-time systems is reviewed. Exploiting this, the solution to the original problem is given in Section \ref{Solution_Conti}, where it is shown in particular that the solution in finite-dimensions is sufficient to deduce the solution for the PDEs. The main results are applied to a boundary controlled vibrating string in Section \ref{Example}. A first perspective could be to enlarge the class of PDEs to a much larger class, allowing for instance for different velocities, i.e. the function $\lambda_0$ could be replaced by a diagonal matrix with different functions on the diagonal. This combined with an additional additive term of the form $P_0(\zeta)x(\zeta,t)$ in \eqref{StateEquation_Diag_Uniform} could be investigated as further research.

\section*{Acknowledgments}
This work was supported by the German Research Foundation (DFG). A. H. is supported by the DFG under the Grant HA 10262/2-1.



\bibliographystyle{abbrvnat}
\bibliography{Biblio}

\end{document}